\documentclass[12pt,a4paper]{article}
\usepackage{amsmath,amssymb,amsfonts,amsthm,mathtools,mathrsfs}
\usepackage{enumitem}
\usepackage[compress]{cite}
\usepackage[dvipsnames]{xcolor}
\usepackage[colorlinks,linkcolor=blue,citecolor=cyan,urlcolor=blue]{hyperref}

\allowdisplaybreaks

\numberwithin{equation}{section}
\newtheorem{theorem}{Theorem}[section]
\newtheorem{lemma}[theorem]{Lemma}

\newtheorem{proposition}[theorem]{Proposition}
\theoremstyle{definition}
\newtheorem{definition}[theorem]{Definition}
\newtheorem{remark}[theorem]{Remark}

\def\e{\varepsilon}
\newcommand{\dd}{\mathrm d}
\newcommand{\Div}{\operatorname{div}}
\newcommand{\Tr}{\operatorname{Tr}}
\newcommand{\R}{\mathbb R}
\newcommand{\N}{\mathbf N}
\newcommand{\norm}[1]{\left\lVert #1\right\rVert}
\newcommand{\Ug}{\mathbf U_{\mathbf g}}
\newcommand{\Hg}{H_{\mathbf g}}
\newcommand{\vF}{\mathbf v_{\mathbf F}}

\begin{document}
\title{On well-posedness theory of very weak solutions to Navier--Stokes equations on irregular domains with nonhomogeneous Dirichlet boundary data}
\author{
		Xiaojin Bai\thanks{ School of Mathematical Sciences, Shanghai Jiao Tong 
			University, Shanghai 200240, China.  Email: baileymoon@sjtu.edu.cn} 
	\and Siran Li \thanks{ Corresponding author at: School of Mathematical Sciences $\&$ CMA-Shanghai, Shanghai Jiao Tong 	University, Shanghai 200240, China.  Email: siran.li@sjtu.edu.cn} 
	\and Xiangxiang Su \thanks{School of Mathematical Sciences, Shanghai Jiao Tong 
		University, Shanghai 200240, China. Email: sjtusxx@sjtu.edu.cn}}
\date{}
\pagestyle{myheadings}
\markboth{ Very weak solutions to the Navier-Stokes equations}
{Very weak solutions to the Navier-Stokes equations}
\maketitle

\begin{abstract}
The well-posedness theory of very weak solutions is a central topic in mathematical hydrodynamics, especially in the regularity theory for Navier--Stokes equations. It has been fully developed for incompressible fluid flows on bounded domains in $\R^3$ of $C^{2,1}$-regularity. In this paper, based on the analytic theories in [D. Breit and A. Gaudin, ArXiv Preprint: 2511.19091  (2025)] and [V.~G. Maz'ya and T.~O. Shaposhnikova, Vol.337, Grundlehren der mathematischen Wissenschaften (2009)], we establish the well-posedness theory of very weak solutions to the Navier--Stokes equations on bounded Lipschitz domains whose boundary has local graphing functions with sufficiently small Sobolev multiplier norm, which contain the bounded Lipschitz domains with sufficiently small Lipschitz constants as a special case.

\end{abstract}

\begin{center}
\begin{minipage}{5.5in}
\end{minipage}
\end{center}

\medskip
\noindent
{\bf Statements and Declarations}.
\\
\noindent
Competing Interests: The authors declare that they have no competing interests.\\
Data Availability: No data was used for the research described in the article.\\
 AI Statement:  The authors thank
 ChatGPT 5.5 AI model  for  computational assistance and language editing. All mathematical derivations, conclusions, and errors remain solely the responsibility of the authors.

\medskip
\noindent
{\bf Acknowledgement}. The authors are indebted to Tobias Barker for insightful discussions on very weak solutions and the partial regularity theory for the Navier--Stokes equations.
The research of SL is supported by NSFC Projects 12331008 $\&$ 12411530065, the Young Elite Scientists Sponsorship Program by CAST 2023QNRC001, National Key Research $\&$ Development Programs 2023YFA1010900 and 2024YFA1014900, Shanghai Rising-Star Program 24QA2703600, Shanghai Qi-Guang Scholarship, and Shanghai Frontiers Science Center of Modern Analysis. 
The research of XS is partially supported by National Key Research $\&$ Development Programs 2023YFA1010900 and 2024YFA1014900.

\section{Introduction}

Consider the instationary Navier--Stokes equations (NSE):
\begin{align}\label{eq:NS}
\left\{
\begin{aligned}
 &\partial_t\mathbf u-\Delta\mathbf u
   +\Div(\mathbf u\otimes\mathbf u)+\nabla \pi
      =\operatorname{div}\mathbf{F}
      &&\text{in }(0,T)\times\Omega,\\
 &\Div\mathbf u=0
      &&\text{in }(0,T)\times\Omega,\\
 &\mathbf u=\mathbf g
      &&\text{on }(0,T)\times\partial\Omega,\\
 &\mathbf u(0)=\mathbf u_0
      &&\text{in }\Omega 
\end{aligned}
\right.
\end{align}
in  a bounded  Lipschitz domain $\Omega\subset\R^3$  with  the functions $\mathbf{F}$, $\mathbf{g}$, $\mathbf{u}_0$ satisfying
	\begin{equation}\label{eq:data}
	\mathbf{F}\in L^s(0,T;L^r(\Omega)), \quad \mathbf{g}\in L^s(0,T;W^{1-\frac1q,q}(\partial\Omega)), \quad \mathbf{u}_0\in L_\sigma^q(\Omega),
	\end{equation}
    where $3<q<\infty$, $2<s<\infty$, $1<r<q$ such that
	$\frac13+\frac1q=\frac1r$ and $ \frac2s+\frac3q=1$. The compatibility condition is imposed: 	\begin{equation}\label{eq:flux}
	\int_{\partial \Omega}\mathbf{g}(t)\cdot \mathbf{N} \,\dd S=0\quad \text{for a.e. }t\in(0,T),
	\end{equation}
	where  $\mathbf{N} = \mathbf{N}(x)$ denotes the outer unit normal vector field along $\partial \Omega$.

In many problems of fluid mechanics, one encounters geometrical or physical data that are too irregular to fit into the classical framework of weak solutions. This motivates Amann (2002--03) to introduce the notion of \emph{very weak solutions} to the NSE~\cite{Amann2002,Amann2003}, which requires the velocity field to be merely $L^s_t L^q_x$-integrable. Since then, very weak solutions have been an important topic in the theory of mathematical hydrodynamics. See Farwig--Galdi--Sohr \cite{FarwigGaldiSohr2005,FarwigGaldiSohr2005b,FarwigGaldiSohr2006a}.

	
	Existence results for very weak solutions to the nonstationary NSE have been established primarily on $C^{2,1}$-domains, \emph{i.e.}, the domains whose boundaries are locally graphs of functions whose second derivatives are Lipschitz continuous;     \emph{cf.} Farwig--Galdi--Sohr~\cite{FarwigGaldiSohr2005, FarwigGaldiSohr2006b}. In this case, the stationary Stokes problem
	$$
	-\Delta \mathbf{u} + \nabla \pi = 0,\quad \operatorname{div} \mathbf{u} = 0,\quad \mathbf{u}|_{\partial\Omega} = \mathbf{g}
	$$
	admits a unique solution $(\mathbf{u},\pi) \in W^{2,q}(\Omega) \times W^{1,q}(\Omega)$ for any $\mathbf{g} \in W^{2-\frac1q,q}(\partial\Omega)$. This follows from standard $W^{2,q}$-elliptic regularity theory, namely the Calder\'{o}n--Zygmund theory. This, however, fails for general Lipschitz domains. Shen \cite{Shen2012}  proved that the Stokes operator generates a bounded analytic semigroup in $L^q_\sigma(\Omega)$ only for $|\frac1q-\frac12|<\frac1{2d}+\varepsilon$ in  Lipschitz
	domains  $\Omega\subset\mathbb{R}^d$ with $d\geq 3$, where $\e$ depends only on $d$  and the Lipschitz character of $\Omega$.  Gabel--Tolksdorf \cite{GabelTolksdorf2022} also considered the case $d=2$.
	 More recently, Geng--Shen~\cite{GengShen2024} obtained $L^q$-resolvent estimates on bounded and exterior $C^1$-domains for all $1<q<\infty$, and further established in \cite{GengShen2024Linfty} $L^\infty$-resolvent estimates in bounded $C^1$-domains for
	$d\geq 3$ and Lipschitz domains for $d=2$. 
	Building upon the resolvent estimates by Shen~\cite{Shen2012} and the maximal regularity of Kunstmann--Weis~\cite{KunstmannWeis2017} for the Stokes operator, Tolksdorf \cite{Tolksdorf2018} developed the $L^q$-theory for the NSE on 3-dimensional bounded Lipschitz domains. 
	As for very weak solutions, Coscia \cite{Coscia2017} established the existence and uniqueness of the solution to stationary NSE in bounded Lipschitz domains, with boundary value in $L^2(\partial\Omega)$. We also refer to \cite{Giga1985Domain,Giga1986Semilinear,GigaMiyakawa1985,GigaSohr1991} by Giga, Miyakawa, and Sohr for pioneering work on the semigroup approach to the NSE.                

The well-known regularity criterion for Leray--Hopf weak solutions to the 3D incompressible NSE \emph{\`{a} la} Lady\u{z}henskaja--Prodi--Serrin has been established for $L^s_tL^q_x$-solutions with $\frac{2}{s}+\frac{3}{q}=1$, which requires  $3<q<\infty$ in the non-endpoint case. On the other hand, in view of Shen~\cite{Shen2012} and subsequent developments, one expects the well-posedness of very weak solutions on 3D bounded Lipschitz domains, where $\frac{3}{2}-\varepsilon < q < 3+\varepsilon$. The above two ranges of the index $q$ are barely disjoint. Hence, we face essential difficulties in developing the well-posedness theory of very weak solutions to NSE on general Lipschitz domains.

Based on recent breakthroughs due to Breit~\cite{Breit2025} and Breit--Gaudin~\cite{BreitGaudin2025} (which, in turn, have theoretical foundations in  Maz'ya--Shaposhnikova~\cite{MS09}), we investigate the subclass of bounded Lipschitz domains whose boundary charts belong to a Sobolev multiplier class with sufficiently small multiplier norm, as detailed in Definition~\ref{def:multiplier-boundary}. This subclass lies strictly between Lipschitz domains and $C^{1,\alpha}$ domains.

We note that the formulation in~\cite{FarwigGaldiSohr2006a} of very weak solutions to NSE on $C^{2,1}$-domains makes crucial use of the expression $$\langle \mathbf g,N\cdot\nabla w\rangle_{\partial\Omega},$$ where $w\in C_0^1([0,T);C_{0,\sigma}^2(\Omega))$ is a solenoidal test function, and $\mathbf{g}\in W^{-1/q,q}(\partial\Omega)$ is the boundary datum. In contrast, for the aforementioned bounded Lipschitz domains in the Sobolev multiplier class, the low regularity of $\partial\Omega$ makes it troublesome to take boundary traces and integrate by parts. We circumvent this issue by formulating the solution in the domain of the Stokes operator $D(A_{q'})$ and interpreting certain terms of divergence form by duality. For smooth divergence-free zero-trace test functions, our definition reduces to the  identity in \cite{FarwigGaldiSohr2006a} (see Remark~\ref{smooth}).


Our main result is as follows. The definition of the Sobolev multiplier class $\mathcal M_W^{1+\alpha,\rho}(\varepsilon)$ is given in Definition~\ref{def:multiplier-boundary}. We refer the reader to Proposition~\ref{prop:linear} for the smallness of $\varepsilon_0$.

\begin{theorem}\label{thm:main}
There exists a universal constant $\varepsilon_0>0$ such that the following holds: Let $\Omega\subset\R^3$ be a bounded  Lipschitz domain of Sobolev multiplier class $\mathcal M_W^{1+\alpha,\rho}(\varepsilon)$ with any $\alpha\in(0,1)$, $\rho\in[1,+\infty]$, and $\varepsilon < \varepsilon_0$. Assume that the data $\mathbf{F}, \mathbf{g}$ and $\mathbf{u}_0$ satisfy the regularity assumptions~\eqref{eq:data} and the compatibility condition~\eqref{eq:flux} with $3<q<\infty$, $2<s<\infty$, $1<r<q$ such that $\frac 13+\frac1q=\frac1r$ and $\frac2s+\frac3q=1$. Then there exists $T^* = T^*(\mathbf{F},\mathbf{g},\mathbf{u}_0) \in (0,T]$ such that the NSE~\eqref{eq:NS} admits a unique very weak solution $\mathbf{u} \in L^s(0, T^* ;L^q (\Omega))$ with 
		\begin{align}\label{eq:global-estimate}
		\norm{\mathbf{u}}_{L^s(0,T;L^q(\Omega))} &\le C\left( \norm{\mathbf{F}}_{L^s(0,T;L^r(\Omega))} +\norm{\mathbf{g}}_{L^s\left(0,T;W^{1-\frac1q,q}(\partial\Omega)\right)} +\norm{\mathbf{u}_0}_{L^q(\Omega)}\right).
		\end{align} 
        The constant $C=C(\Omega,\alpha,\rho,q)$ is independent of $T$.

Moreover, there exists a positive constant $\mu=\mu(\Omega,\alpha,\rho,q)$  independent of $T$ such that the following holds: Suppose that the data satisfy the smallness condition
\begin{equation}\label{eq:global-smallness}
		\norm{\mathbf{F}}_{L^s(0,T;L^r(\Omega))} +\norm{\mathbf{g}}_{L^s(0,T;W^{1-\frac1q,q}(\partial\Omega))}+\norm{\mathbf{u}_0}_{L^q(\Omega)}\leq \mu \qquad \text{for any } T>0.
		\end{equation} 
    Then there exists a unique very weak solution $\mathbf{u} \in L^s(0, \infty ;L^q (\Omega))$ to the NSE~\eqref{eq:NS} that satisfies the estimate~\eqref{eq:global-estimate}. 
	\end{theorem}


\begin{remark}
Theorem~\ref{thm:main} can be adapted to the inhomogeneous Besov-space setting \(B_{p,q}^{s}(\Omega)\) instead of \(L^q(\Omega)\), with \(1<p<\infty\), \(1\leq q\leq\infty\), and \(s\in(-1+\frac 1p,\frac 1p)\). 
Indeed, replacing the  estimates in Proposition~\ref{prop:linear} by their Besov analogues from \cite{BreitGaudin2025}, the same argument yields the existence and uniqueness of a very weak solution \(\mathbf{u}\in L^{\sigma}(0,T; B_{p,q}^{s}(\Omega))\) to the NSE~\eqref{eq:NS} with $2<\sigma<\infty$ and data in the  corresponding Besov spaces.
\end{remark}

To outline our strategies for the proof of Theorem~\ref{thm:main}, we recall again the essential technical difficulty in our current setting: the low regularity of both the domain and the boundary data make, in general, the classical \(W^{2,q}\)-estimates for the Stokes system invalid. To overcome this issue, we  construct stationary auxiliary fields that absorb the forcing and boundary data, and then solve the linearised problem in our very weak setting by making use of Stokes semigroup estimates, which have been established for the irregular domains in \cite{BreitGaudin2025}. By interpreting suitable nonlinear terms via duality, we succeed in reformulating the definition of very weak solutions to the NSE as the solution (\emph{a.k.a.} the mild solution) to an equivalent integral equation. The resulting integral equation can be  solved effectively through fixed point arguments.

The paper is organized as follows. In Section~\ref{sec2} we define the multiplier boundary class and collect the Stokes estimates, which will be used throughout the paper. In  Section~\ref{sec3} we construct stationary auxiliary fields and introduces the very weak solution. In Section~\ref{sec4}, we establish the existence and uniqueness of very weak solutions to the linear Stokes problem. Section~\ref{sec5} establishes the equivalence between the very weak formulation and an integral equation. The proof of the Main Theorem~\ref{thm:main} is presented in Section~\ref{sec:proof-main}. In addition, we sketch the proof of the Stokes estimates in Appendix~\ref{app:stokes-toolbox}.

\section{Preliminaries}\label{sec2}

Let $\Omega \subset \mathbb{R}^3$ be a bounded Lipschitz domain. For $1<q<\infty$, we denote by $q'$ its conjugate exponent, i.e., $\frac1q+\frac1{q'}=1$. The Lebesgue space $L^q(\Omega)$ is equipped with the norm $\|\cdot\|_{L^q(\Omega)}$, and  the Sobolev space $W^{\alpha,q}(\Omega)$ is equipped  with the norm $\|\cdot\|_{W^{\alpha,q}(\Omega)}$ for $\alpha\geq 0$. The space $W^{-\alpha,q}(\Omega):=(W_0^{\alpha,q'}(\Omega))^*$ denotes the dual space of $W_0^{\alpha,q'}(\Omega)$, with the natural duality pairing denoted by $\langle\cdot,\cdot\rangle_{\Omega}$. The same notation applies to functions defined on the boundary $\partial\Omega$, namely $L^q(\partial\Omega)$, $W^{\alpha,q}(\partial\Omega)$, and $W^{-\alpha,q}(\partial\Omega)$. We shall make use of the standard local spaces $L_{\text{loc}}^q(\Omega)$ and $W_{\text{loc}}^{\alpha,q}(\Omega)$, defined as the sets of measurable functions belonging to $L^q(K)$ and $W^{\alpha,q}(K)$, respectively, for every compact $K\subset\Omega$.

The space of compactly supported smooth functions is denoted by $C_c^\infty(\Omega)$, and we set
$$
C_{c,\sigma}^\infty(\Omega) := \{ \mathbf{u} \in C_c^\infty(\Omega) : \operatorname{div} \mathbf{u} = 0 \}.
$$
The space $L_\sigma^q(\Omega)$ is then defined as the closure of $C_{c,\sigma}^\infty(\Omega)$ in $L^q(\Omega)$. The space $W_{0,\sigma}^{1,q}(\Omega)$ is defined as the closure of $C_{c,\sigma}^\infty(\Omega)$ in $W_0^{1,q}(\Omega)$, or equivalently,
$$
W_{0,\sigma}^{1,q}(\Omega) := \{ \mathbf{u} \in W_0^{1,q}(\Omega) : \operatorname{div} \mathbf{u} = 0 \}.
$$

	\subsection{Boundary regularity via Sobolev multiplier}

	We first recall the notion of Sobolev multipliers from~\cite[Chapter 14]{MS09}. For a function $\varphi \in W_{\text{loc}}^{1,1}(\Omega)$, the multiplier norm for Sobolev spaces $W^{1+\alpha,\rho}(\Omega)$ is defined by
	\begin{equation*}
	\|\varphi\|_{\mathcal{M}_W^{1+\alpha,\rho}(\Omega)} := \sup_{\substack{\mathbf{v} \in W^{\alpha,\rho}(\Omega, \mathbb{C}^3) \\ \|\mathbf{v}\|_{W^{\alpha,\rho}(\Omega)} = 1}} \|\nabla \varphi \cdot \mathbf{v}\|_{W^{\alpha,\rho}(\Omega)}
	\end{equation*}
	with $\alpha \in \mathbb{R}$ and $\rho \in [1,\infty]$. 
	
	\begin{definition}\label{def:multiplier-space}
		The \textit{Sobolev multiplier space} $\mathcal{M}_W^{1+\alpha,\rho}(\Omega)$ is the set of all functions $\varphi \in W_{\text{loc}}^{1,1}(\Omega)$ for which the norm $\|\varphi\|_{\mathcal{M}_W^{1+\alpha,\rho}(\Omega)}$ is finite. For $\varepsilon > 0$, we denote by $\mathcal{M}_W^{1+\alpha,\rho}(\Omega, \varepsilon)$ the open ball of radius $\varepsilon$ in $\mathcal{M}_W^{1+\alpha,\rho}(\Omega)$, i.e.,
		$$
		\mathcal{M}_W^{1+\alpha,\rho}(\Omega, \varepsilon) := \bigl\{ \varphi \in \mathcal{M}_W^{1+\alpha,\rho}(\Omega) : \|\varphi\|_{\mathcal{M}_W^{1+\alpha,\rho}(\Omega)} < \varepsilon \bigr\}.
		$$
		In the two-dimensional case, the space $\mathcal{M}_W^{1+\alpha,\rho}(\mathbb{R}^2)$ is defined analogously by replacing $\Omega$ with $\mathbb{R}^2$ and taking $\mathbf{v} \in \mathbb{C}^2$.
	\end{definition}
	
We now impose a regularity condition on the boundary following \cite[Section 3]{Breit2024}. 

\begin{definition}\label{def:multiplier-boundary}
	Let $\Omega \subset \mathbb{R}^3$ be a bounded domain, $\alpha \in (0,1)$, $\rho \in [1,\infty]$, and $\varepsilon > 0$ sufficiently small. We say that \textit{$\partial\Omega$ belongs to the multiplier class $\mathcal{M}_W^{1+\alpha,\rho}(\varepsilon)$} if there exists a finite cover of $\partial\Omega$ by open sets $U^1,\ldots,U^\ell$ for some $\ell\in \mathbb{N}$ such that, for each $j$, after a suitable translation and rotation,
	$$
	U^j \cap \Omega = U^j \cap \bigl\{ (\mathbf{x}', x_3) \in \mathbb{R}^3 : x_3 > \varphi_j(\mathbf{x}') \bigr\}
	$$
	for some $\varphi_j:\mathbb{R}^2\to\mathbb{R}$, and
	$$
	\varphi_j \in \mathcal{M}_W^{1+\alpha,\rho}(\mathbb{R}^2, \varepsilon) \qquad \text{for all } j = 1, \ldots, \ell.
	$$
\end{definition}
	


\subsection{The Stokes--Dirichlet operator}

In what follows, let  $\Omega \subset \mathbb{R}^3$  be a bounded Lipschitz domain whose boundary belongs to the multiplier class $\mathcal{M}_W^{1+\alpha,\rho}(\varepsilon)$ for  $\alpha\in(0,1)$, $\rho\in[1,\infty]$, and sufficiently small $\varepsilon>0$.

 For $1<q<\infty$, the Stokes--Dirichlet
operator $A_q$ is defined by
\begin{align}\label{eq:stokes-action}
 A_q\mathbf v:=\mathbf f, \qquad \mathbf v\in D(A_q),
\end{align}
where  $D(A_q)$ consists of all $\mathbf v \in W_{0,\sigma}^{1,q}(\Omega)$ for which there exists  $\mathbf f \in L_\sigma^q(\Omega)$ satisfying
\begin{align*}
\int_\Omega \nabla \mathbf v : \nabla \boldsymbol\phi \, \dd x
= \int_\Omega \mathbf f \cdot \boldsymbol\phi \, \dd x
\quad \text{for all } \boldsymbol\phi \in W_{0,\sigma}^{1,q'}(\Omega).
\end{align*}
The  norm on $D(A_q)$ is given by
$$
\|\mathbf v\|_{D(A_q)} := \|\mathbf v\|_{L^q(\Omega)} + \|A_q\mathbf v\|_{L^q(\Omega)}.
$$

The vector field $\mathbf f$ is unique, since
$W_{0,\sigma}^{1,q'}(\Omega)$ is dense in $L_\sigma^{q'}(\Omega)$ (see Proposition \eqref{prop:linear}  (i) below).
In the present  setting, we emphasize that $D(A_q)$ is not in general identified with $W^{2,q}(\Omega)\cap W_0^{1,q}(\Omega)\cap L_\sigma^q(\Omega)$. All the Stokes estimates needed in the sequel are collected in the next proposition, with its proof deferred to Appendix~\ref{app:stokes-toolbox}.

Let $3<q<\infty$ be fixed and let $r$ be defined by
$
 \frac1r=\frac13+\frac1q.
$
Set
$
 \mathcal E_q:=\{q,q',r,r'\}
$. 
\begin{proposition}\label{prop:linear} 
For every $p\in\mathcal E_q$, there exists $\varepsilon>0$ sufficiently small in
Definition~\ref{def:multiplier-boundary} such that the following assertions
hold on  bounded Lipschitz domains of class
$\mathcal M_W^{1+\alpha,\rho}(\varepsilon)$.

\begin{enumerate}[label=\textup{(\roman*)},leftmargin=*]
\item The Helmholtz projection
$\mathbb P_p:L^p(\Omega)\to L_\sigma^p(\Omega)$ is bounded and
\begin{align}\label{eq:Hodge}
 L^p(\Omega) =
 L_\sigma^p(\Omega)\oplus
 \nabla\bigl(W^{1,p}(\Omega)/\R\bigr).
\end{align}
Moreover, for $\mathbf f\in L^p(\Omega)$ and $\boldsymbol\psi\in L^{p'}(\Omega)$
\begin{align}\label{eq:projection-duality}
 \int_\Omega\mathbb P_p\mathbf f\cdot\boldsymbol\psi\,\dd x
 =
 \int_\Omega\mathbf f\cdot\mathbb P_{p'}\boldsymbol\psi\,\dd x.
\end{align}
The duality $\bigl(L_\sigma^p(\Omega)\bigr)^* \cong L_\sigma^{p'}(\Omega)$ holds via the integral pairing: for every bounded linear functional $\Lambda \in \bigl(L_\sigma^{p'}(\Omega)\bigr)^*$, there exists a unique $\mathbf f \in L_\sigma^p(\Omega)$ such that
\[
\Lambda(\boldsymbol\psi) = \int_\Omega \mathbf f \cdot \boldsymbol\psi \, dx, \qquad \forall \boldsymbol\psi \in L_\sigma^{p'}(\Omega).
\]

\item The operator $A_p$ defined by
\eqref{eq:stokes-action} is
invertible on $L_\sigma^p(\Omega)$ and satisfies
\begin{align}\label{eq:square-root}
 D(A_p^{1/2})=W_{0,\sigma}^{1,p}(\Omega),
 \qquad
 \norm{A_p^{1/2}\mathbf v}_{L^p(\Omega)}
 \simeq\norm{\nabla\mathbf v}_{L^p(\Omega)}.
\end{align}
For $\mathbf v\in W_{0,\sigma}^{1,p}(\Omega)$ and
$\boldsymbol\phi\in W_{0,\sigma}^{1,p'}(\Omega)$,
\begin{align}\label{eq:square-root-form}
 \int_\Omega\nabla\mathbf v:\nabla\boldsymbol\phi\,\dd x
 =
 \int_\Omega A_p^{1/2}\mathbf v\cdot
 A_{p'}^{1/2}\boldsymbol\phi\,\dd x.
\end{align}
The inverse and the semigroup satisfy
\begin{align}\label{eq:resolvent-duality}
 \int_\Omega A_p^{-1}\mathbf f\cdot\boldsymbol\psi\,\dd x
 =
 \int_\Omega\mathbf f\cdot A_{p'}^{-1}\boldsymbol\psi\,\dd x,
\end{align}
and
\begin{align}\label{eq:semigroup-duality}
 \int_\Omega e^{-tA_p}\mathbf f\cdot\boldsymbol\psi\,\dd x
 =
 \int_\Omega\mathbf f\cdot e^{-tA_{p'}}\boldsymbol\psi\,\dd x,
 \qquad t\ge0.
\end{align}

\item \label{item:MR}
For  $1<\tau<\infty$,
$T\in(0,\infty]$ and
$\mathbf h\in L^\tau(0,T;L_\sigma^p(\Omega))$, the problem
\begin{align*}
 \partial_t\mathbf z+A_p\mathbf z=\mathbf h,
 \qquad
 \mathbf z(0)=0
\end{align*}
has a unique solution
$
 \mathbf z\in W^{1,\tau}(0,T;L_\sigma^p(\Omega))
 \cap L^\tau(0,T;D(A_p)),
$
and
\begin{align}\label{eq:MR}
 \norm{\mathbf z_t}_{L^\tau(0,T;L_\sigma^p(\Omega))}
 +\norm{A_p\mathbf z}_{L^\tau(0,T;L_\sigma^p(\Omega))}
 \le C\norm{\mathbf h}_{L^\tau(0,T;L_\sigma^p(\Omega))},
\end{align}
where $C$ is independent of $T$.

\item \label{item:exp}
There are $M_p\ge1$ and
$\omega_p>0$ such that
\begin{align}\label{eq:exp}
 \norm{e^{-tA_p}}_{\mathcal L(L_\sigma^p(\Omega))}
 \le M_pe^{-\omega_pt},
 \qquad t\ge0,
\end{align}
where $\mathcal{L}(L_\sigma^p(\Omega))$ denotes the  space of bounded linear operators on  $L_\sigma^p(\Omega)$. In particular, for $p=q$ there are $C,\omega>0$ such that
\begin{align}\label{eq:cross-exp}
 \norm{\nabla e^{-tA_{q'}}\boldsymbol\psi}_{L^{(\frac q2)'}(\Omega)}
 \le Ct^{-\frac12-\frac{3}{2q}}e^{-\omega t}
 \norm{\boldsymbol\psi}_{L^{q'}(\Omega)},
 \qquad t>0.
\end{align}

\end{enumerate}
\end{proposition}

\section{ Definition of very weak solution}\label{sec3}
In this section, we construct stationary auxiliary fields that allow us to formulate the very weak problem without invoking $W^{2,q}$-regularity of the Stokes system. These fields absorb, respectively, the divergence-form force term and the boundary datum, thereby circumventing the integrations by parts that are available only on smoother domains. 

\begin{proposition}\label{prop:force-field}
For every $\mathbf F\in L^r(\Omega)$, there is a unique
$\vF\in W_{0,\sigma}^{1,r}(\Omega)$ satisfying
\begin{align}\label{eq:force-weak}
 \int_\Omega\nabla\vF:\nabla\boldsymbol\phi\,\dd x
 =-\int_\Omega\mathbf F:\nabla\boldsymbol\phi\,\dd x
 \qquad
 \forall\boldsymbol\phi\in W_{0,\sigma}^{1,r'}(\Omega).
\end{align}
Moreover, $\vF\in L_\sigma^q(\Omega)$ and
\begin{align}\label{eq:force-est}
 \norm{\vF}_{W^{1,r}(\Omega)}+
 \norm{\vF}_{L^q(\Omega)}
 \le C\norm{\mathbf F}_{L^r(\Omega)}.
\end{align}
\end{proposition}

\begin{proof}
We begin by establishing a general solvability result for the weak Stokes problem. 
For any $p\in\mathcal{E}_q$ and 
$\ell\in \bigl(W^{1,p'}_{0,\sigma}(\Omega)\bigr)^*$, 
there exists a unique $\mathbf{v}\in W^{1,p}_{0,\sigma}(\Omega)$ such that
\begin{equation}\label{weakS}
\int_{\Omega}\nabla \mathbf{v}:\nabla \boldsymbol{\phi}\,dx 
= \ell(\boldsymbol{\phi}), 
\qquad \forall \boldsymbol{\phi}\in W^{1,p'}_{0,\sigma}(\Omega). 
\end{equation}
Indeed, by Proposition \ref{prop:linear} (i), the dual space of $L^p_\sigma(\Omega)$ is isometrically 
isomorphic to $L^{p'}_\sigma(\Omega)$ via the integral pairing. Consequently, there exists 
a unique $\mathbf{f}\in L^p_\sigma(\Omega)$ such that
\[
\ell\bigl(A_p^{-1/2}\boldsymbol{\psi}\bigr) 
= \int_{\Omega}\mathbf{f}\cdot \boldsymbol{\psi}\,dx,
\qquad \forall \boldsymbol{\psi}\in L^{p'}_\sigma(\Omega),
\]
with the bound
\[
\|\mathbf{f}\|_{L^p(\Omega)} 
\leq C \|\ell\|_{\bigl(W^{1,p'}_{0,\sigma}(\Omega)\bigr)^*}.
\]
Define $\mathbf{v}:=A_p^{-1/2}\mathbf{f}$. Then $\mathbf{v}\in W^{1,p}_{0,\sigma}(\Omega)$ 
by  Proposition \ref{prop:linear} (ii). 
For any $\boldsymbol{\phi}\in W^{1,p'}_{0,\sigma}(\Omega)$, set 
$\boldsymbol{\psi}:=A_{p'}^{1/2}\boldsymbol{\phi}$. 
Applying  \eqref{eq:square-root-form}, we obtain
\[
\int_{\Omega}\nabla \mathbf{v}:\nabla \boldsymbol{\phi}\,dx
= \int_{\Omega} A_p^{1/2}\mathbf{v}\cdot A_{p'}^{1/2}\boldsymbol{\phi}\,dx
= \int_{\Omega}\mathbf{f}\cdot \boldsymbol{\psi}\,dx
= \ell(\boldsymbol{\phi}),
\]
which proves \eqref{weakS}. Uniqueness follows immediately from the invertibility of $A_p$ 
on $L^p_\sigma(\Omega)$.

We now apply this general result with $p=r$ and define
\[
\ell_{\mathbf{F}}(\boldsymbol{\phi}) 
:= -\int_{\Omega}\mathbf{F}:\nabla \boldsymbol{\phi}\,dx,
\qquad \boldsymbol{\phi}\in W^{1,r'}_{0,\sigma}(\Omega).
\]
Since $\mathbf{F}\in L^r(\Omega)$ and $\nabla\boldsymbol{\phi}\in L^{r'}(\Omega)$, 
H\"older's inequality yields $\ell_{\mathbf{F}}\in \bigl(W^{1,r'}_{0,\sigma}(\Omega)\bigr)^*$. 
The preceding argument therefore furnishes a unique 
$\mathbf{v}_{\mathbf{F}}\in W^{1,r}_{0,\sigma}(\Omega)$ satisfying \eqref{eq:force-weak}, together with 
the estimate
\begin{equation}\label{vf1}
\|\mathbf{v}_{\mathbf{F}}\|_{W^{1,r}(\Omega)} 
\leq C\|\mathbf{F}\|_{L^r(\Omega)}. 
\end{equation}

Since $r<3$ and $\frac{1}{q}=\frac{1}{r}-\frac{1}{3}$, the Sobolev embedding 
$W^{1,r}(\Omega)\hookrightarrow L^q(\Omega)$ is continuous. Then
\begin{equation}\label{vf2}
\|\mathbf{v}_{\mathbf{F}}\|_{L^q(\Omega)} 
\leq C\|\mathbf{v}_{\mathbf{F}}\|_{W^{1,r}(\Omega)} 
\leq C\|\mathbf{F}\|_{L^r(\Omega)}. 
\end{equation}
Hence, \eqref{vf1} and \eqref{vf2} yield \eqref{eq:force-est}.

Finally, to verify that $\mathbf{v}_{\mathbf{F}}$ belongs to $L^q_\sigma(\Omega)$, 
we invoke the standard density result for solenoidal functions on bounded Lipschitz domains: there exists a sequence 
$\{\boldsymbol{\phi}_n\}_{n=1}^{\infty}\subset C^\infty_{c,\sigma}(\Omega)$ such that 
$\boldsymbol{\phi}_n\to \mathbf{v}_{\mathbf{F}}$ in $W^{1,r}(\Omega)$. 
By the same Sobolev embedding, convergence also holds in $L^q(\Omega)$. 
Hence $\mathbf{v}_{\mathbf{F}}$ lies in the closure of $C^\infty_{c,\sigma}(\Omega)$ 
in $L^q(\Omega)$, i.e., $\mathbf{v}_{\mathbf{F}}\in L^q_\sigma(\Omega)$. 
This completes the proof. 
\end{proof}


The next proposition constructs a stationary extension of the boundary data that enjoys a homogeneous variational identity in the interior. This property is essential because a merely solenoidal extension would not produce the boundary cancellation used later.

\begin{proposition}\label{prop:boundary-extension}
Let $\mathbf g\in W^{1-\frac1q,q}(\partial\Omega)$ satisfy the compatibility condition \eqref{eq:flux}.
Then there exists a unique vector field $\mathbf U_{\mathbf g}\in W^{1,q}(\Omega)$ such that
\begin{align}\label{eq:Ua-properties}
 \Div\Ug=0,
 \qquad
 \Tr\Ug=\mathbf g,
 \qquad
 \int_\Omega\nabla\Ug:\nabla\boldsymbol\phi\,\dd x=0
 \quad
 \forall\boldsymbol\phi\in W_{0,\sigma}^{1,q'}(\Omega),
\end{align}
and the estimate
\begin{align}\label{eq:Ua-est}
 \norm{\Ug}_{W^{1,q}(\Omega)}
 \le C\norm{\mathbf g}_{W^{1-\frac1q,q}(\partial\Omega)}.
\end{align}
Consider the Hodge decomposition
\begin{align}\label{eq:Ua-Hodge}
 \Ug=\mathbb P_q\Ug+\nabla\Hg,
\end{align}
 where $\Hg$  is defined modulo constants. 
 Then
\begin{align}\label{eq:Ha-est}
 \norm{\mathbb P_q\Ug}_{L^q(\Omega)}
 +\norm{\nabla\Hg}_{L^q(\Omega)}
 \le C\norm{\mathbf g}_{W^{1-\frac1q,q}(\partial\Omega)},
\end{align}
and  for every $\phi\in W^{1,q'}(\Omega)$,
\begin{align}\label{eq:Ha-flux}
 \int_\Omega\nabla\Hg\cdot\nabla\phi\,\dd x
 =\int_{\partial\Omega}(\mathbf g\cdot\N)\Tr\phi\,\dd S.
\end{align}
\end{proposition}

\begin{proof}
Let
$E_0:W^{1-\frac1q,q}(\partial\Omega)\to W^{1,q}(\Omega)$
be a bounded right inverse of the trace operator. The compatibility condition
\eqref{eq:flux} and the divergence theorem lead to
\begin{align*}
 \int_\Omega\Div(E_0\mathbf g)\,\dd x
 =\int_{\partial\Omega}\mathbf g\cdot\N\,\dd S=0.
\end{align*}
By Bogovskii's theorem \cite{Galdi2011}, there exists
$\mathbf b_{\mathbf g}\in W_0^{1,q}(\Omega)$ such that
\begin{align*}
 \Div\mathbf b_{\mathbf g}=\Div(E_0\mathbf g),
 \qquad
 \norm{\mathbf b_{\mathbf g}}_{W^{1,q}(\Omega)}
 \le C\norm{\Div(E_0\mathbf g)}_{L^q(\Omega)}.
\end{align*}
Define
\begin{align*}
 \mathbf E_{\mathbf g}:=E_0\mathbf g-\mathbf b_{\mathbf g}.
\end{align*}
Then 
\begin{align*}
 \Div\mathbf E_{\mathbf g}=0,
 \qquad
 \Tr\mathbf E_{\mathbf g}=\mathbf g,
 \qquad
 \norm{\mathbf E_{\mathbf g}}_{W^{1,q}(\Omega)}
 \le C\norm{\mathbf g}_{W^{1-\frac1q,q}(\partial\Omega)}.
\end{align*}
Using Proposition~\ref{prop:force-field},
there exists a unique
$\mathbf z_{\mathbf g}\in W_{0,\sigma}^{1,q}(\Omega)$ such that
\begin{align*}
 \int_\Omega\nabla\mathbf z_{\mathbf g}:
 \nabla\boldsymbol\phi\,\dd x
 =
 \int_\Omega\nabla\mathbf E_{\mathbf g}:
 \nabla\boldsymbol\phi\,\dd x
 \qquad
 \forall\boldsymbol\phi\in W_{0,\sigma}^{1,q'}(\Omega),
\end{align*}
with
\begin{align*}
 \norm{\mathbf z_{\mathbf g}}_{W^{1,q}(\Omega)}
 \le C\norm{\mathbf E_{\mathbf g}}_{W^{1,q}(\Omega)}.
\end{align*}
Define
\begin{align*}
 \Ug:=\mathbf E_{\mathbf g}-\mathbf z_{\mathbf g}.
\end{align*}
Then $\Ug$ satisfies \eqref{eq:Ua-properties} and \eqref{eq:Ua-est} by construction. 
Suppose $\mathbf U_1$ and $\mathbf U_2$ both satisfy \eqref{eq:Ua-properties}. Then $\mathbf U_1-\mathbf U_2\in W_{0,\sigma}^{1,q}(\Omega)$ and
\[
\int_\Omega \nabla \mathbf (\mathbf U_1-\mathbf U_2):\nabla \boldsymbol\phi\,dx=0
\qquad \forall \boldsymbol\phi\in W_{0,\sigma}^{1,q'}(\Omega).
\]
It follows from the uniqueness assertion in Proposition \ref{prop:force-field} with $\mathbf F=0$ that $\mathbf U_1=\mathbf U_2$.

The Hodge decomposition \eqref{eq:Ua-Hodge} follows from
Proposition~\ref{prop:linear} \textup{(i)}. The estimate \eqref{eq:Ha-est} is then a consequence of the boundedness of $\mathbb P_q$ and \eqref{eq:Ua-est}. Since
$\mathbb P_q\Ug\in L_\sigma^q(\Omega)$, one has
\begin{align*}
 \int_\Omega\mathbb P_q\Ug\cdot\nabla\phi\,\dd x=0
 \qquad
 \forall\phi\in W^{1,q'}(\Omega).
\end{align*}
On the other hand, since $\Div\Ug=0$ and $\Tr\Ug=\mathbf g$,  Green's formula gives
\begin{align*}
 \int_\Omega\Ug\cdot\nabla\phi\,\dd x
 =\int_{\partial\Omega}(\mathbf g\cdot\N)\Tr\phi\,\dd S.
\end{align*}
Combining the last two identities with
$\Ug=\mathbb P_q\Ug+\nabla\Hg$ yields  \eqref{eq:Ha-flux}.
\end{proof}

For later use, we record the following identity for the boundary extension:
    if $\boldsymbol\Phi$ is smooth in a neighborhood of $\overline\Omega$ such that  $\Div\boldsymbol\Phi=0$  and $\Tr\boldsymbol\Phi=0$, then
\begin{align}\label{eq:Ua-transposition}
 -\int_\Omega\Ug\cdot\Delta\boldsymbol\Phi\,\dd x
 +\int_{\partial\Omega}
 \mathbf g\cdot\partial_{\N}\boldsymbol\Phi\,\dd S=0,
\end{align}
where 
 $\partial_{\N}\boldsymbol\Phi:=(\nabla\boldsymbol\Phi)\N$. 
Indeed, $\boldsymbol\Phi\in W^{1,q'}_{0,\sigma}(\Omega)$, and \eqref{eq:Ua-properties} gives
$$
 0= \int_\Omega\nabla\Ug:\nabla\boldsymbol\Phi\,\dd x .
$$
Since $\Tr\Ug=\mathbf g$, then integration by parts yields
\eqref{eq:Ua-transposition}.

We now define very weak solutions of the Navier--Stokes problem
\eqref{eq:NS}.

\begin{definition}\label{def:solution}
 Assume that the data $\mathbf{F}, \mathbf{g}$ and $\mathbf{u}_0$ satisfy  the regularity assumptions \eqref{eq:data} and compatibility condition \eqref{eq:flux}. Let
$\vF$, $\Ug$  and $\Hg$ be the  stationary fields from
Proposition~\ref{prop:force-field} and Proposition~\ref{prop:boundary-extension}. A vector field $\mathbf u\in L^s(0,T;L^q(\Omega))$ is a \textit {very weak solution} of \eqref{eq:NS} if for every $\boldsymbol\Phi\in\mathcal C_c^1([0,T);D(A_{q'}))$,
\begin{align}\label{eq:transposition-graph}
\begin{aligned}
 &-\int_0^T\!\!\int_\Omega
 \mathbf u\cdot\partial_t\boldsymbol\Phi\,\dd x\dd t
 +\int_0^T\!\!\int_\Omega
 (\mathbb P_q\mathbf u-\mathbb P_q\Ug)
 \cdot A_{q'}\boldsymbol\Phi\,\dd x\dd t
 -\int_0^T\!\!\int_\Omega
 (\mathbf u\otimes\mathbf u):\nabla\boldsymbol\Phi\,\dd x\dd t\\
 &=\int_\Omega\mathbf u_0\cdot\boldsymbol\Phi(0)\,\dd x
 +\int_0^T\!\!\int_\Omega
 \vF\cdot A_{q'}\boldsymbol\Phi\,\dd x\dd t,
\end{aligned}
\end{align}
and for every $\phi\in C_c^\infty([0,T)\times\overline\Omega)$,
\begin{align}\label{eq:scalar-transposition}
 -\int_0^T\!\!\int_\Omega\mathbf u\cdot\nabla\phi\,\dd x\dd t
 +\int_0^T\!\!\int_{\partial\Omega}
 (\mathbf g\cdot\N)\phi\,\dd S\dd t=0.
\end{align}
 
The linear very weak formulation is obtained by deleting the nonlinear term.
\end{definition}

We now consider the test space $C_c^1([0,T);D(A_{q'}))$  in Definition \ref{def:solution}. 
 The nonlinear term is well-defined for this test space. In
fact, if $\boldsymbol\Phi\in D(A_{q'})$, then
\begin{align}\label{eq:graph-gradient}
 \norm{\nabla\boldsymbol\Phi}_{L^{(\frac q2)'}(\Omega)}
 =\norm{\nabla A_{q'}^{-1}A_{q'}\boldsymbol\Phi}_{L^{(\frac q2)'}(\Omega)}
 \le C\norm{A_{q'}\boldsymbol\Phi}_{L^{q'}(\Omega)},
\end{align}
where the estimate follows from  integrating  \eqref{eq:cross-exp} over $(0,\infty)$ for $\psi=A_{q'}\boldsymbol\Phi$.

We  next claim  that \eqref{eq:scalar-transposition} is equivalent to the system
\[
\Div \mathbf u=0
\qquad
\mathbf u\cdot\mathbf N=\mathbf g\cdot\mathbf N
\quad\text{on } \partial\Omega. 
\]
To this end, 
taking  $\phi(t,x)=\eta(t)\zeta(x)$ with $\eta\in C_c^\infty(0,T)$ and $\zeta\in C^\infty(\overline\Omega)$ in \eqref{eq:scalar-transposition} gives
\[
\int_0^T \eta(t)\int_\Omega \mathbf u(t,x)\cdot \nabla \zeta(x)\,\dd x\dd t
=
\int_0^T \eta(t)\int_{\partial\Omega} (\mathbf g\cdot \mathbf N)\zeta(x)\,\dd S \dd t.
\]
Applying  \eqref{eq:Ha-flux}, we obtain
\[
\int_0^T \eta(t)
\int_\Omega \bigl(\mathbf u(t,x)-\nabla H_{\mathbf g}(t,x)\bigr)\cdot \nabla \zeta(x)\,\dd x \dd t=0,
\]
Since $W^{1,q'}(\Omega)$ is separable, we choose a
countable dense family $\{\zeta_k\}_{k=1}^\infty$ in
$W^{1,q'}(\Omega)$. 
For each fixed $k$, the arbitrariness of $\eta\in C_c^\infty(0,T)$ implies that
\[
\int_\Omega \bigl(\mathbf u(t)-\nabla H_{\mathbf g}(t)\bigr)\cdot\nabla\zeta_k
\,\dd x =0
\qquad \text{for a.e. } t\in(0,T),
\]
where the exceptional null set may depend on $k$. Taking the union of these countably many exceptional sets yields a single null set $N\subset(0,T)$ such that
\begin{equation}\label{id}
\int_\Omega \bigl(\mathbf u(t)-\nabla H_{\mathbf g}(t)\bigr)\cdot\nabla\zeta_k
\,\dd x =0
\qquad \forall k,\  t\notin N.
\end{equation}
Since $\{\zeta_k\}$ is dense in $W^{1,q'}(\Omega)$ and the mapping
\[
\zeta \longmapsto \int_\Omega \bigl(\mathbf u(t)-\nabla H_{\mathbf g}(t)\bigr)\cdot \nabla \zeta\,\dd x
\]
is continuous for each fixed $t\notin N$, then \eqref{id} extends to all $\zeta\in W^{1,q'}(\Omega)$ for $t\notin N$. Hence, for almost every $t\in(0,T)$,
\[
\int_\Omega \bigl(\mathbf u(t)-\nabla H_{\mathbf g}(t)\bigr)\cdot \nabla\zeta\,\dd x=0
\qquad \forall \zeta \in W^{1,q'}(\Omega).
\]
The Helmholtz decomposition \eqref{eq:Hodge} now yields
\begin{align}\label{eq:gradient-component}
(I-\mathbb P_q)\mathbf u = \nabla \Hg
\qquad \text{for a.e. } t\in(0,T).
\end{align}
Conversely, combining \eqref{eq:gradient-component} with \eqref{eq:Ha-flux} recovers \eqref{eq:scalar-transposition}. Thus the two formulations are equivalent.

\begin{remark}\label{smooth}
We show that for any
\[
\boldsymbol\Phi\in C_c^1([0,T);C^\infty(\overline\Omega)),\qquad
\Div\boldsymbol\Phi=0,\qquad
\Tr\boldsymbol\Phi=0,
\]
the identity \eqref{eq:transposition-graph}  reduces to 
\begin{align}\label{eq:smooth-transposition}
&-\int_0^T\!\!\int_\Omega
\mathbf u\cdot(\partial_t\boldsymbol\Phi
+\Delta\boldsymbol\Phi)\,\dd x\dd t
-\int_0^T\!\!\int_\Omega
(\mathbf u\otimes\mathbf u):\nabla\boldsymbol\Phi\,\dd x\dd t
+\int_0^T\!\!\int_{\partial\Omega}
\mathbf g\cdot\partial_{\N}\boldsymbol\Phi\,\dd S\dd t\notag\\
&=\int_\Omega\mathbf u_0\cdot\boldsymbol\Phi(0)\,\dd x
-\int_0^T\!\!\int_\Omega
\mathbf F:\nabla\boldsymbol\Phi\,\dd x\dd t,
\end{align}	
 which is precisely the 
very weak  formulation in \cite{FarwigGaldiSohr2006a}  when $\Div\mathbf u=0$ and   $\mathbf{g}\in W^{1-\frac1q,q}(\partial\Omega)$.
	
	To verify this, we first note that if $\boldsymbol\Phi$ is smooth in a neighborhood of $\overline\Omega$ with $\Div\boldsymbol\Phi=0$ and $\Tr\boldsymbol\Phi=0$, then
	\begin{align}\label{eq:smooth-domain}
	\boldsymbol\Phi\in D(A_{q'}),
	\qquad
	A_{q'}\boldsymbol\Phi=\mathbb P_{q'}(-\Delta \boldsymbol\Phi).
	\end{align}
	Indeed, for any $\boldsymbol\eta\in W_{0,\sigma}^{1,q'}(\Omega)$, integration by parts gives
	\[
	\int_\Omega\nabla\boldsymbol\Phi:\nabla\boldsymbol\eta\,\dd x
	=\int_\Omega(-\Delta\boldsymbol\Phi)\cdot\boldsymbol\eta\,\dd x
	=\int_\Omega\mathbb P_{q'}(-\Delta\boldsymbol\Phi)\cdot\boldsymbol\eta\,\dd x.
	\]
Now, since
$$
 (I-\mathbb P_q)\mathbf u=\nabla\Hg
 \quad\text{and}\quad
 (I-\mathbb P_q)\Ug=\nabla\Hg,
$$
we have
$$
 \mathbb P_q(\mathbf u-\Ug)=\mathbf u-\Ug .
$$
Using this identity  together with \eqref{eq:smooth-domain}  and \eqref{eq:projection-duality} for $p=q'$, we obtain
\begin{align*}
 \int_\Omega
 (\mathbb P_q\mathbf u-\mathbb P_q\Ug)
 \cdot A_{q'}\boldsymbol\Phi\,\dd x
 =\int_\Omega
 \mathbb P_q(\mathbf u-\Ug)
 \cdot \mathbb P_{q'}(-\Delta\boldsymbol\Phi)\,\dd x
 =-\int_\Omega
 (\mathbf u-\Ug)\cdot\Delta\boldsymbol\Phi\,\dd x .
\end{align*}
Moreover, since $\vF\in L_\sigma^q(\Omega)$ and
$A_{q'}\boldsymbol\Phi=\mathbb P_{q'}(-\Delta\boldsymbol\Phi)$,
 \eqref{eq:force-weak} gives
\begin{align*}
 \int_\Omega\vF\cdot A_{q'}\boldsymbol\Phi\,\dd x
 =\int_\Omega\vF\cdot(-\Delta\boldsymbol\Phi)\,\dd x
 =\int_\Omega\nabla\vF:\nabla\boldsymbol\Phi\,\dd x
 = -\int_\Omega\mathbf F:\nabla\boldsymbol\Phi\,\dd x .
\end{align*}
Finally,  it follows from \eqref{eq:Ua-transposition} that
\begin{align*}
 -\int_\Omega\Ug\cdot\Delta\boldsymbol\Phi\,\dd x
 +\int_{\partial\Omega}\mathbf g\cdot
 \partial_{\N}\boldsymbol\Phi\,\dd S=0.
\end{align*}
Substituting the last three identities into  \eqref{eq:transposition-graph} yields \eqref{eq:smooth-transposition}.
\end{remark}


\section{The linear problem }\label{sec4}
In this section, we establish the existence and uniqueness of very weak
 solutions to the linearized Navier--Stokes system  introduced in Definition~\ref{def:solution}. More precisely, we consider  the Stokes system 
\begin{equation}\label{linear}
\left\{
\begin{aligned}
\partial_t \mathbf u - \Delta \mathbf u + \nabla \pi &= \operatorname{div} \mathbf F &&\text{in } (0,T)\times\Omega,\\
\operatorname{div} \mathbf u &= 0 &&\text{in } (0,T)\times\Omega,\\
\mathbf u &= \mathbf g &&\text{on } (0,T)\times\partial\Omega,\\
\mathbf u(0) &= \mathbf u_0 &&\text{in } \Omega,
\end{aligned}
\right.
\end{equation}
where $\mathbf F,\mathbf g,\mathbf u_0$ satisfy the regularity assumptions \eqref{eq:data} and compatibility condition \eqref{eq:flux}.

\begin{definition}\label{def:linear-solution}
	Assume that $\mathbf F,\mathbf g,\mathbf u_0$ satisfy \eqref{eq:data} and \eqref{eq:flux}. Let $\mathbf v_{\mathbf F}$ and $\mathbf U_{\mathbf g}$ be the stationary auxiliary fields from Propositions~\ref{prop:force-field} and~\ref{prop:boundary-extension}. A vector field $\mathbf u\in L^s(0,T;L^q(\Omega))$ is a  \textit{very weak solution} of \eqref{linear} if
 for every $\boldsymbol\Phi\in C_c^1([0,T);D(A_{q'}))$,
		\begin{equation}\label{l1}
		\begin{aligned}
		&-\int_0^T\int_\Omega \mathbf u\cdot \partial_t \boldsymbol\Phi\,\dd x\dd t
		+\int_0^T\int_\Omega (\mathbb P_q\mathbf u-\mathbb P_q\mathbf U_{\mathbf g})
		\cdot A_{q'}\boldsymbol\Phi\,\dd x\dd t\\
		&\qquad=
		\int_\Omega \mathbf u_0\cdot \boldsymbol\Phi(0)\,\dd x
		+\int_0^T\int_\Omega \mathbf v_{\mathbf F}\cdot A_{q'}\boldsymbol\Phi\,\dd x\dd t;
		\end{aligned}
			\end{equation}
 and for every $\phi\in C_c^\infty([0,T)\times\overline\Omega)$,
		\begin{equation}\label{l2}
		-\int_0^T\int_\Omega \mathbf u\cdot \nabla \phi\,\dd x\dd t
		+\int_0^T\int_{\partial\Omega} (\mathbf g\cdot \mathbf N)\phi\,\dd S\dd t=0.
		\end{equation}
\end{definition}
Now, we define
\begin{align}\label{eq:MT}
 \mathfrak M_T(\mathbf u_0,\mathbf F,\mathbf g)
 :=\;&\norm{e^{-tA_q}\mathbf u_0}_{L^s(0,T;L^q(\Omega))}
 +\norm{\mathbf F}_{L^s(0,T;L^r(\Omega))} +\norm{\mathbf g}_{L^s(0,T;
 W^{1-\frac1q,q}(\partial\Omega))}.
\end{align}
\begin{proposition}\label{prop:linear-solution}
	Assume that $\mathbf F,\mathbf g,\mathbf u_0$ satisfy \eqref{eq:data} and \eqref{eq:flux}. 
 Then the system \eqref{linear} admits a unique very weak solution $\overline{\mathbf U}\in L^s(0,T;L^q(\Omega))$. Moreover, there exists a constant $C^*>0$, independent of $T$, such that
\begin{align}\label{eq:linear-est}
 \norm{\overline{\mathbf U}}_{L^s(0,T;L^q(\Omega))}
 \le C^* \mathfrak M_T.
\end{align} 
\end{proposition}

\begin{proof}
Set $$\mathbf h:=\vF+\mathbb P_q\Ug.$$
By estimates \eqref{eq:force-est}, \eqref{eq:Ua-est}, \eqref{eq:Ha-est}, and the boundedness of $\mathbb P_q$ in  Proposition~\ref{prop:linear}, we have 
\begin{align}\label{eq:h-est}
\begin{aligned}
 \norm{\mathbf h}_{L^s(0,T;L^q(\Omega))}
 \le C\Big(&\norm{\mathbf F}_{L^s(0,T;L^r(\Omega))}+\norm{\mathbf g}_{L^s(0,T;
 W^{1-\frac1q,q}(\partial\Omega))}\Big),
\end{aligned}
\end{align}
where $C$ is independent of $T$.

It follows from Proposition~\ref{prop:linear} (iii) that the Cauchy problem
\begin{align}\label{eq:z}
 \partial_t\mathbf z+A_q\mathbf z=\mathbf h, \qquad \mathbf z(0)=0
\end{align}
has a unique solution
$
 \mathbf z\in W^{1,s}(0,T;L_\sigma^q(\Omega))  \cap L^s(0,T;D(A_q)).
$
Define
\begin{align*}
 \overline{\mathbf Y}(t)
 :=A_q^{-1}e^{-tA_q}\mathbf u_0+\mathbf z(t)
\end{align*}
and
\begin{align}\label{eq:Ubar}
 \overline{\mathbf U}(t)
 :=A_q\overline{\mathbf Y}(t)+\nabla\Hg(t)
 =e^{-tA_q}\mathbf u_0+A_q\mathbf z(t)+\nabla\Hg(t).
\end{align}

We verify that $\overline{\mathbf U}$ satisfies \eqref{l1} and \eqref{l2} in Definition~\ref{def:linear-solution}. First,
since $\mathbb P_q\nabla\Hg=0$, we deduce from \eqref{eq:Ubar} that
\begin{align}\label{eq:Ybar-projection}
 \overline{\mathbf Y}=A_q^{-1}\mathbb P_q\overline{\mathbf U}.
\end{align}
Moreover,  the term $A_q^{-1}e^{-tA_q}\mathbf u_0$ satisfies the homogeneous equation
\begin{align*}
 \partial_t(A_q^{-1}e^{-tA_q}\mathbf u_0)
 +e^{-tA_q}\mathbf u_0=0,
\end{align*}
Combining this with the equation \eqref{eq:z}, we obtain
\begin{align}\label{eq:linear-evolution}
 \partial_t\overline{\mathbf Y}+
 \mathbb P_q\overline{\mathbf U}=\mathbf h
 =\vF+\mathbb P_q\Ug,
 \qquad
 \overline{\mathbf Y}(0)=A_q^{-1}\mathbf u_0.
\end{align}
Let $\boldsymbol\Phi\in\mathcal C_c^1([0,T);D(A_{q'}))$. Pairing \eqref{eq:linear-evolution} with $A_{q'}\boldsymbol\Phi$, integrating in time, and using   \eqref{eq:resolvent-duality} together with the fact that $\boldsymbol\Phi_t$ is solenoidal,  yields \eqref{l1} in Definition~\ref{def:linear-solution}. The identity \eqref{l2} follows directly from 
 the formula \eqref{eq:Ubar}, since the first two terms on the right-hand side of \eqref{eq:Ubar} are solenoidal and the remaining term \(\nabla H_{\mathbf g}\) satisfies \eqref{eq:Ha-flux}. Thus $\overline{\mathbf U}$ is a very weak solution of \eqref{linear} in the sense of Definition~\ref{def:linear-solution}.

For uniqueness, let $\mathbf u^*$ be the difference of two very weak solutions. 
Applying \eqref{l2} to $\mathbf u^*$ yields the corresponding identity with vanishing boundary term. In view of  \eqref{eq:gradient-component}, it implies
\[
(I-\mathbb P_q)\mathbf u^*(t)=\nabla H_{\mathbf g}(t)=0
\qquad \text{for a.e. } t\in(0,T),
\]
where the last equality follows from \eqref{eq:Ha-flux} with $\mathbf g=0$. Hence $\mathbf u^*=\mathbb P_q\mathbf u^*$.
 Now let
$\boldsymbol\Psi\in C_c^1([0,T);L_\sigma^{q'}(\Omega))$ and take
$\boldsymbol\Phi=A_{q'}^{-1}\boldsymbol\Psi$ in 
\eqref{l1} for the difference. Setting
$\mathbf Y:=A_q^{-1}\mathbf u^*$, one obtains
\begin{align*}
 -\int_0^T\!\!\int_\Omega
 \mathbf Y\cdot\partial_t\boldsymbol\Psi\,\dd x\dd t
 +\int_0^T\!\!\int_\Omega
 A_q\mathbf Y\cdot\boldsymbol\Psi\,\dd x\dd t=0,
 \qquad
 \mathbf Y(0)=0.
\end{align*}
Thus $ \partial_t\mathbf Y+A_q\mathbf Y=0$ in the sense of distributions.  Since
$A_q\mathbf Y=\mathbf u^*\in L^s(0,T;L_\sigma^q(\Omega))$, we have $\partial_t\mathbf Y\in L^s_{\mathrm{loc}}(0,T;L_\sigma^q(\Omega))$ and hence
$\mathbf Y\in W^{1,s}_{\mathrm{loc}}(0,T;L_\sigma^q(\Omega))$. The uniqueness of the  Cauchy problem in Proposition~\ref{prop:linear} (iii) gives $\mathbf Y=0$,  and consequently $\mathbf u^*=0$.

It remains to prove the bound \eqref{eq:linear-est}. From  \eqref{eq:Ubar}, we have
\[
\|\overline{\mathbf U}\|_{L^s(0,T;L^q(\Omega))}
\leq
\|e^{-tA_q}\mathbf u_0\|_{L^s(0,T;L^q(\Omega))}
+\|A_q\mathbf z\|_{L^s(0,T;L^q(\Omega))}
+\|\nabla H_{\mathbf g}\|_{L^s(0,T;L^q(\Omega))}.
\]
The first term is part of $\mathfrak M_T$ in \eqref{eq:MT}. For the second term, \eqref{eq:MR} and \eqref{eq:h-est} yield
\[
\|A_q\mathbf z\|_{L^s(0,T;L^q(\Omega))}
\leq C\bigl(\|\mathbf F\|_{L^s(0,T;L^r(\Omega))}+\|\mathbf g\|_{L^s(0,T;W^{1-\frac1q,q}(\partial\Omega))}\bigr),
\]
where $C$ is independent of $T$.
The third term is bounded by  $\mathfrak M_T$ via \eqref{eq:Ha-est}, again with a constant independent of $T$. Combining these estimates gives \eqref{eq:linear-est}.

\end{proof}


\section{The nonlinear term and the integral equation}\label{sec5}

In this section, we transform the  very weak formulation from Definition~\ref{def:solution} into an equivalent integral equation in the space $L^s(0,T;L^q(\Omega))$. This expresses the solution as a fixed point of a nonlinear operator, which  will be the key to the existence proof.

 The main difficulty is that the nonlinear term $\Div(\mathbf u\otimes \mathbf u)$ is not an $L^q$-function in space, since $\mathbf u\otimes \mathbf u$ belongs only to $L^{\frac q2}(\Omega)$. Consequently, the expression $A_q^{-1}\mathbb P_q \Div(\mathbf u\otimes \mathbf u)$ cannot be interpreted directly. To overcome this, we define this term by duality, using the inverse Stokes operator and the analytic semigroup.

We begin with the stationary operator $\mathcal R_q$, which will replace the formal expression $A_q^{-1}\mathbb P_q \Div \mathbf G$ for rough tensor fields $\mathbf G$. 
 For $\boldsymbol\psi\in L_\sigma^{q'}(\Omega)$, 
 the inverse semigroup formula 
 and \eqref{eq:cross-exp} give 
\begin{align}\label{eq:resolvent-gradient}
 \norm{\nabla A_{q'}^{-1}\boldsymbol\psi}_{L^{(\frac q2)'}(\Omega)}
 &\le\int_0^\infty
 \norm{\nabla e^{-tA_{q'}}\boldsymbol\psi}_{L^{(\frac q2)'}(\Omega)}\,\dd t
 \le C\norm{\boldsymbol\psi}_{L^{q'}(\Omega)}.
\end{align}
For  $\mathbf G\in L^{\frac q2}(\Omega)$,  define
$\mathscr R_q\mathbf G\in L_\sigma^q(\Omega)$ by
\begin{align}\label{eq:Rq}
 \int_\Omega\mathscr R_q\mathbf G\cdot\boldsymbol\psi\,\dd x
 =-\int_\Omega\mathbf G:\nabla A_{q'}^{-1}
 \boldsymbol\psi\,\dd x,
 \qquad
 \forall\ \boldsymbol\psi\in L_\sigma^{q'}(\Omega).
\end{align}
This is well defined by \eqref{eq:resolvent-gradient}.  Moreover,
\begin{align}\label{eq:Rq-est}
 \norm{\mathscr R_q\mathbf G}_{L^q(\Omega)}
 \le C\norm{\mathbf G}_{L^{\frac q2}(\Omega)}.
\end{align}
If $\textbf{G}\in C_c^\infty(\Omega)$, then by \eqref{eq:projection-duality}, \eqref{eq:resolvent-duality},  \eqref{eq:Rq}  and integration by parts,
\begin{align}\label{eq:smooth-Rq}
 \mathscr{R}_q\mathbf G=A_q^{-1}\mathbb P_q\Div\mathbf G
\end{align}

Next, we introduce the time-dependent operator $\mathcal K$, which incorporates the evolution semigroup and will be used to define the nonlinear term in the integral formulation.
For $\mathbf G \in L^{\frac s2}(0, T ; L^{\frac q2}(\Omega))$, define $\mathcal K\mathbf G(t)$ by
\begin{align}\label{eq:K}
 \int_\Omega\mathcal K\mathbf G(t)\cdot\boldsymbol\psi\,\dd x
 :=-\int_0^t\!\!\int_\Omega
 \mathbf G(\tau):\nabla e^{-(t-\tau)A_{q'}}
 \boldsymbol\psi\,\dd x\dd\tau, \quad \text{a.e.}\  t\in(0,T).
\end{align}

The following lemma extends this definition to $L^{\frac s2}(0,T;L^{\frac q2}(\Omega))$, the space in which $\mathbf u\otimes\mathbf u$ lies when $\mathbf u\in L^s(0,T;L^q(\Omega))$.

\begin{lemma}\label{lem:K}
The map $\mathcal K$ is bounded from
$L^{\frac s2}(0,T;L^{\frac q2}(\Omega))$ to
$L^s(0,T;L_\sigma^q(\Omega))$ with
\begin{align}\label{eq:K-est}
 \norm{\mathcal K\mathbf G}_{L^s(0,T;L^q(\Omega))}
 \le C_K
 \norm{\mathbf G}_{L^{\frac s2}(0,T;L^{\frac q2}(\Omega))},
\end{align}
where the constant $C_K>0$ is independent of $T$.
\end{lemma}

\begin{proof}
    Set $\gamma:=\frac12+\frac{3}{2q}$. 
   Since $q>3$, we have $\gamma<1$. Using H\"older's inequality and \eqref{eq:cross-exp}, we obtain
\begin{align*}
 \norm{\mathcal K\mathbf G(t)}_{L^q(\Omega)}
 \le C\int_0^t(t-\tau)^{-\gamma}
 \norm{\mathbf G(\tau)}_{L^{\frac q2}(\Omega)}\,\dd\tau.
\end{align*}
 Extending $\|\mathbf G(\cdot)\|_{L^{q/2}(\Omega)}$ by zero outside $(0,T)$, we apply the one-dimensional Hardy--Littlewood--Sobolev inequality \cite[Chapter~4]{LiebLoss2001} of order $\frac 1s$, which maps $L^{\frac s2}(\mathbb R)$ into $L^s(\mathbb R)$ since $1<\frac s2<s$. This implies \eqref{eq:K-est}. The constant is independent of $T$ because of the zero extension.
\end{proof}

For $\mathbf w,\mathbf v\in L^s(0,T;L^q(\Omega))$, define
\begin{align*}
 \mathcal Q(\mathbf w,\mathbf v)
 :=\mathcal K(\mathbf w\otimes\mathbf v).
\end{align*}
Lemma~\ref{lem:K}  and H\"older's inequality  yield that 
\begin{align}\label{eq:Q-est}
 \norm{\mathcal Q(\mathbf w,\mathbf v)}_{L^s(0,T;L^q(\Omega))}
 \le C_B
 \norm{\mathbf w}_{L^s(0,T;L^q(\Omega))}
 \norm{\mathbf v}_{L^s(0,T;L^q(\Omega))} 
\end{align}
with $C_B$ independent of $T$.

The next lemma shows that $\mathcal K\mathbf G$  is exactly $A_q\mathbf Y$, where $\mathbf Y$ is the unique solution of the evolution equation with source $\mathscr R_q\mathbf G$ and zero initial data by Proposition~\ref{prop:linear} (iii). 

\begin{lemma}\label{lem:bridge}
Let
$\mathbf G\in L^{\frac s2}(0,T;L^{\frac q2}(\Omega))$  and  $\mathbf Y\in W^{1,\frac s2}(0,T;L_\sigma^q(\Omega))
 \cap L^{\frac s2}(0,T;D(A_q))$
be the unique solution of
\begin{align}\label{eq:Y-G}
 \partial_t\mathbf Y+A_q\mathbf Y=\mathscr R_q\mathbf G,
 \qquad
 \mathbf Y(0)=0.
\end{align}
Then
\begin{align}\label{eq:AY-K}
 A_q\mathbf Y=\mathcal K\mathbf G
 \qquad\text{a.e. on }(0,T),
\end{align}
and
\begin{align}\label{eq:K-operator}
 \partial_t(A_q^{-1}\mathcal K\mathbf G)
 +\mathcal K\mathbf G=\mathscr R_q\mathbf G,
 \qquad
 (A_q^{-1}\mathcal K\mathbf G)(0)=0.
\end{align}
\end{lemma}
\begin{proof}
Choose a sequence
$\mathbf G_n\in C_c^\infty((0,T)\times\Omega)$ such that
$\mathbf G_n\to\mathbf G$ in
$L^{\frac s2}(0,T;L^{\frac q2}(\Omega))$ as $n\to\infty$.
If $T=\infty$, take the approximants with compact time support. 
Let
$\mathbf Y_n$ be the unique solution of
\begin{align*}
 \partial_t\mathbf Y_n+A_q\mathbf Y_n=\mathscr R_q\mathbf G_n,
 \qquad
 \mathbf Y_n(0)=0.
\end{align*}
By Duhamel's formula, 
together with \eqref{eq:smooth-Rq}, we have
\begin{align*}
 A_q\mathbf Y_n(t)
 &=\int_0^t e^{-(t-\tau)A_q}
 \mathbb P_q\Div\mathbf G_n(\tau)\,\dd\tau.
\end{align*}
Pairing this expression with $\boldsymbol\psi\in L_\sigma^{q'}(\Omega)$ and using  \eqref{eq:semigroup-duality}, we obtain  \eqref{eq:K} for $\mathcal K\mathbf G_n$. Hence
\begin{align*}
 A_q\mathbf Y_n=\mathcal K\mathbf G_n.
\end{align*}
 It follows from \eqref{eq:MR} and \eqref{eq:Rq-est} that
\begin{align*}
 A_q\mathbf Y_n\to A_q\mathbf Y
 \quad\text{in }L^{\frac s2}(0,T;L_\sigma^q(\Omega)) 
\end{align*}
as $n\to\infty$. On the other hand, Lemma~\ref{lem:K} gives
\begin{align*}
 \mathcal K\mathbf G_n\to\mathcal K\mathbf G
 \quad\text{in }L^s(0,T;L_\sigma^q(\Omega)).
\end{align*}
Since \(L^s(0,T_0)\hookrightarrow L^{\frac s2}(0,T_0)\) on every finite interval \((0,T_0)\subset(0,T)\), the two limits of the common sequence \(A_q\mathbf Y_n=\mathcal K\mathbf G_n\) coincide on \((0,T_0)\).  This proves \eqref{eq:AY-K} on every finite subinterval.
If $T=\infty$, the same conclusion follows by exhaustion. 
Substituting \eqref{eq:AY-K} into \eqref{eq:Y-G} and using invertibility of $A_q$ gives $\mathbf Y=A_q^{-1}\mathcal K\mathbf G$ and
\begin{align*}
		\partial_t\mathbf Y+\mathcal K\mathbf G=\mathscr R_q\mathbf G,\qquad \mathbf Y(0)=0.
		\end{align*}
		which is \eqref{eq:K-operator}.
\end{proof}

Now we are in position to establish the equivalence between the  very weak formulation in Definition~\ref{def:solution} with an equivalent integral equation.


\begin{proposition}\label{prop:mild}
Assume that $\mathbf F,\mathbf g,\mathbf u_0$ satisfy \eqref{eq:data} and \eqref{eq:flux}. Then $\mathbf u\in L^s(0,T;L^q(\Omega))$ is a very weak solution
of NSE \eqref{eq:NS} if and only if it satisfies 
\begin{align}\label{eq:mild}
 \mathbf u=\overline{\mathbf U}-\mathcal Q(\mathbf u,\mathbf u)
 \qquad\text{in }L^s(0,T;L^q(\Omega)),
\end{align}
where $\overline{\mathbf U}$ is given by \eqref{eq:Ubar}.
\end{proposition}

\begin{proof}
Assume that $\mathbf u\in L^s(0,T;L^q(\Omega))$ is a very weak solution of NSE \eqref{eq:NS}.
From  \eqref{eq:Ubar}, since both $e^{-tA_q}\mathbf u_0$ and $A_q\mathbf z$ take values in $L_\sigma^q(\Omega)$, we have
\begin{align*}
 (I-\mathbb P_q)\overline{\mathbf U}=\nabla\Hg.
\end{align*}
Then it follows from  \eqref{eq:gradient-component} that \((I-\mathbb P_q)\mathbf u=(I-\mathbb P_q)\overline{\mathbf U}\). 
 Let
$\boldsymbol\Psi\in C_c^1([0,T);L_\sigma^{q'}(\Omega))$ and take
$\boldsymbol\Phi=A_{q'}^{-1}\boldsymbol\Psi$ in  \eqref{eq:transposition-graph}.  For
$\mathbf Y=A_q^{-1}\mathbb P_q\mathbf u$, it derives from  \eqref{eq:resolvent-duality} and  \eqref{eq:Rq} that
\begin{small}
\begin{align*}
 -\int_0^T\!\!\int_\Omega
 \mathbf Y\cdot\partial_t\boldsymbol\Psi\,\dd x\dd t
 +\int_0^T\!\!\int_\Omega
 \bigl(\mathbb P_q\mathbf u-\mathbb P_q\Ug
 +\mathscr R_q(\mathbf u\otimes\mathbf u)-\vF\bigr)
 \cdot\boldsymbol\Psi\,\dd x\dd t= \int_\Omega A_q^{-1}\mathbf u_0\cdot
 \boldsymbol\Psi(0)\,\dd x.
\end{align*}
\end{small}
Consequently, on every finite subinterval,
\begin{align}\label{eq:operator-from-transposition}
 \partial_t\mathbf Y+\mathbb P_q\mathbf u
 =\vF+\mathbb P_q\Ug
 -\mathscr R_q(\mathbf u\otimes\mathbf u),
 \qquad
 \mathbf Y(0)=A_q^{-1}\mathbf u_0.
\end{align}
Since all terms on the right-hand side belong locally to
$L^{\frac s2}(0,T;L_\sigma^q(\Omega))$, we get
$\mathbf Y\in W_{\mathrm{loc}}^{1,\frac s 2}(0,T;L_\sigma^q(\Omega))$  and the initial datum‌ in \eqref{eq:operator-from-transposition} is well-defined.
Subtracting \eqref{eq:linear-evolution} from
\eqref{eq:operator-from-transposition} and applying
Lemma~\ref{lem:bridge} with
$\mathbf G=\mathbf u\otimes\mathbf u$, then
the vector field
\begin{align*}
 \mathbf Z:=A_q^{-1}\bigl(
 \mathbb P_q\mathbf u-
 \mathbb P_q\overline{\mathbf U}
 +\mathcal Q(\mathbf u,\mathbf u)\bigr)
\end{align*}
satisfies
\begin{align*}
 \partial_t \mathbf Z+A_q\mathbf Z=0,
 \qquad
 \mathbf Z(0)=0.
\end{align*}
By uniqueness of the Cauchy problem in Proposition~\ref{prop:linear} (iii),  $\mathbf Z=0$. 
Since $A_q$ is injective,
\begin{align*}
 \mathbb P_q\mathbf u=
 \mathbb P_q\overline{\mathbf U}-\mathcal Q(\mathbf u,\mathbf u).
\end{align*}
Together with
$(I-\mathbb P_q)\mathbf u=(I-\mathbb P_q)\overline{\mathbf U}=\nabla\Hg$,
we obtain
\begin{align*}
 \mathbf u
 &=\mathbb P_q\mathbf u+(I-\mathbb P_q)\mathbf u =\mathbb P_q\overline{\mathbf U}-\mathcal Q(\mathbf u,\mathbf u)
 +(I-\mathbb P_q)\overline{\mathbf U}
 =\overline{\mathbf U}-\mathcal Q(\mathbf u,\mathbf u),
\end{align*}
which implies \eqref{eq:mild}.

Conversely, assume that  \eqref{eq:mild} holds. Since
$\mathcal Q(\mathbf u,\mathbf u)=\mathcal K(\mathbf u\otimes\mathbf u)$ takes
values in $L_\sigma^q(\Omega)$, we have
\begin{align*}
 (I-\mathbb P_q)\mathbf u=(I-\mathbb P_q)\overline{\mathbf U}=\nabla\Hg.
\end{align*}
Combining this identity with \eqref{eq:Ha-flux} and integrating in time yields
 \eqref{eq:scalar-transposition}.

It remains to prove \eqref{eq:transposition-graph}. Put $\mathbf G=\mathbf u\otimes\mathbf u$ and
$\mathbf X=A_q^{-1}\mathcal K\mathbf G$. From \eqref{eq:mild},
\begin{align*}
 A_q^{-1}\mathbb P_q\mathbf u
 =\overline{\mathbf Y}-\mathbf X.
\end{align*}
Applying \eqref{eq:linear-evolution} to $\overline{\mathbf Y}$ and \eqref{eq:K-operator} to $\mathbf X$, we get
\begin{align*}
 \partial_t(A_q^{-1}\mathbb P_q\mathbf u)
 +\mathbb P_q\mathbf u
 =\vF+\mathbb P_q\Ug-
 \mathscr R_q(\mathbf u\otimes\mathbf u),
 \qquad
 (A_q^{-1}\mathbb P_q\mathbf u)(0)=A_q^{-1}\mathbf u_0.
\end{align*}
For $\boldsymbol\Phi\in\mathcal C_c^1([0,T);D(A_{q'}))$, testing this equation with $A_{q'}\boldsymbol\Phi$ and integrating in
time gives \eqref{eq:transposition-graph}. 
In this process, the terms involving the time derivative and the initial data are treated using \eqref{eq:resolvent-duality}. The remaining term is handled by the definition of $\mathscr R_q$ in \eqref{eq:Rq}, which gives
\begin{align*}
 \int_\Omega\mathscr R_q(\mathbf u\otimes\mathbf u)
 \cdot A_{q'}\boldsymbol\Phi\,\dd x
 =-
 \int_\Omega(\mathbf u\otimes\mathbf u):\nabla\boldsymbol\Phi\,\dd x.
\end{align*}
Thus $\mathbf u$ satisfies both identities in Definition~\ref{def:solution},
and then is a very weak solution.
\end{proof}

\section{Proof of Theorem \ref{thm:main}}\label{sec:proof-main}

This section is devoted to the proof of Theorem~\ref{thm:main}. We first establish an existence result under the assumption that $\mathfrak M_T$
is sufficiently small. This assumption is then verified in two scenarios: $T$ small, and for all $T>0$ under the additional smallness condition~\eqref{eq:global-smallness}.


\begin{proposition}\label{prop:prescribed}
 	Assume that the data $\mathbf{F}, \mathbf{g}$ and $\mathbf{u}_0$ satisfy  \eqref{eq:data} and  \eqref{eq:flux}. There exists $\mu_0=\mu_0(\Omega,\alpha,\rho,q)>0$, independent of $T$, such that if
\begin{align*}
 \mathfrak M_T(\mathbf u_0,\mathbf F,\mathbf g)\le\mu_0,
\end{align*}
then NSE \eqref{eq:NS} admits a unique very weak solution for any $T\in(0,\infty]$.
Moreover,
\begin{align*}
 \norm{\mathbf u}_{L^s(0,T;L^q(\Omega))}
 \le C\mathfrak M_T(\mathbf u_0,\mathbf F,\mathbf g),
\end{align*}
where $C$ is independent of $T$.
\end{proposition}
    \begin{proof}
By Proposition~\ref{prop:linear-solution}, one has
\begin{align*}
 \norm{\overline{\mathbf U}}_{L^s(0,T;L^q(\Omega))}\le C^*\mathfrak M_T(\mathbf u_0,\mathbf F,\mathbf g).
\end{align*}
Write $\mathbf u=\overline{\mathbf U}+\mathbf V$.  Then the  equation
\eqref{eq:mild} is equivalent to
\begin{align*}
 \mathbf V=\mathcal T\mathbf V
 &:=-\mathcal Q(\overline{\mathbf U}+\mathbf V,
 \overline{\mathbf U}+\mathbf V).
\end{align*}
Let
$
 \delta_T:=\norm{\overline{\mathbf U}}_{L^s(0,T;L^q(\Omega))}
$
and consider the closed ball
\begin{align*}
 \mathbb B_{\delta_T}:=
 \left\{\mathbf V\in L^s(0,T;L^q(\Omega)):
 \norm{\mathbf V}_{L^s(0,T;L^q(\Omega))}\le\delta_T\right\}.
\end{align*}
For $\mathbf V\in\mathbb B_{\delta_T}$, the bilinear estimate \eqref{eq:Q-est} yields
\begin{align*}
 \norm{\mathcal T\mathbf V}_{L^s(0,T;L^q(\Omega))}
 \le4C_B\delta_T^2.
 \end{align*}
Similarly, for $\mathbf V_1,\mathbf V_2\in\mathbb B_{\delta_T}$,
\begin{align*}
 \norm{\mathcal T\mathbf V_1-
 \mathcal T\mathbf V_2}_{L^s(0,T;L^q(\Omega))}
 \le4C_B\delta_T
 \norm{\mathbf V_1-\mathbf V_2}_{L^s(0,T;L^q(\Omega))}.
\end{align*}
If $\delta_T\le \frac 1{8C_B}$, then $\mathcal T$ maps $\mathbb B_{\delta_T}$ into itself
and is a contraction with Lipschitz constant at most $\frac 12$. 
Banach's fixed-point theorem therefore gives a unique fixed point $\mathbf V\in\mathbb B_{\delta_T}$. By Proposition~\ref{prop:mild}, the  function $\mathbf u=\overline{\mathbf U}+\mathbf V$ is a very weak solution of \eqref{eq:NS}, and
\[
\|\mathbf u\|_{L^s(0,T;L^q(\Omega))}
\le 2\delta_T
\le 2C^*\mathfrak M_T(\mathbf u_0,\mathbf F,\mathbf g).
\]
Hence we may take $\mu_0=(8C_BC^*)^{-1}$.

It remains to prove uniqueness without assuming that another
solution lies in the contraction ball. Let
$\mathbf u_1,\mathbf u_2\in L^s(0,T;L^q(\Omega))$ be very weak solutions with the same
data, and set $\mathbf  u^*=\mathbf u_1-\mathbf u_2$. By
Proposition~\ref{prop:mild},
\begin{align}\label{eq:difference}
 \mathbf u^*=-\mathcal Q(\mathbf u^*,\mathbf u_1)
 -\mathcal Q(\mathbf u_2,\mathbf u^*).
\end{align}
For $T<\infty$, we argue by partition. Since $\mathbf u_1,\mathbf u_2\in L^s(0,T;L^q(\Omega))$, the function
\[
t\mapsto\int_0^t\left(
\|\mathbf u_1(\tau)\|_{L^q(\Omega)}^s
+\|\mathbf u_2(\tau)\|_{L^q(\Omega)}^s\right)d\tau
\]
is absolutely continuous. Thus there exists a finite partition
\[
0=t_0<t_1<\cdots<t_N=T
\]
such that, with $I_j=(t_{j-1},t_j)$,
\[
C_B\left(
\|\mathbf u_1\|_{L^s(I_j;L^q(\Omega))}
+\|\mathbf u_2\|_{L^s(I_j;L^q(\Omega))}\right)<1
\qquad \forall j=1,\dots,N.
\]

On $I_1$, the definition of $\mathcal K$ only involves times $0<\tau<t$, so applying \eqref{eq:Q-est} to \eqref{eq:difference} gives
\[
\|\mathbf u^*\|_{L^s(I_1;L^q(\Omega))}
\le C_B\left(
\|\mathbf u_1\|_{L^s(I_1;L^q(\Omega))}
+\|\mathbf u_2\|_{L^s(I_1;L^q(\Omega))}\right)
\|\mathbf u^*\|_{L^s(I_1;L^q(\Omega))}.
\]
Since the factor in parentheses is strictly less than $1$, it follows that $\mathbf u^*=0$ on $I_1$.
Now suppose $\mathbf u^*=0$ on $(0,t_{j-1})$. For $t\in I_j$, the time integrals defining $\mathcal Q(\mathbf u^*,\mathbf u_1)(t)$ and $\mathcal Q(\mathbf u_2,\mathbf u^*)(t)$ over $0<\tau<t_{j-1}$ vanish. Hence only the integrals over $t_{j-1}<\tau<t$ remain. Applying the same estimate on $I_j$ (after translating to an interval starting at zero) yields $\mathbf u^*=0$ on $I_j$. By induction, $\mathbf u^*=0$ on $(0,T)$.

For $T=\infty$, the same argument applies to every finite interval, yielding uniqueness.
\end{proof}

Now we are in position to prove Theorem~\ref{thm:main}.

\medskip

\noindent\textit{Proof of Theorem~\ref{thm:main}.}
First, we establish the local-in-time existence of very weak solution. Choose \(T^*\in(0,\min\{T,1\}]\), with the convention \(\min\{\infty,1\}=1\). As \(t\downarrow 0\), due to the boundedness of the analytic semigroup, 
\[
\|e^{-\tau A_q}\mathbf u_0\|_{L^s(0,t;L^q(\Omega))}
\le C t^{\frac 1s}\|\mathbf u_0\|_{L^q(\Omega)}
\]
shows that this initial term in \(\mathfrak M_t\) decays to zero. The force and boundary terms in \(\mathfrak M_t\) also tend to zero as \(t\downarrow 0\) by absolute continuity of the Bochner integrals. Hence there exists \(T^*>0\) sufficiently small such that \(\mathfrak M_{T^*}\le \mu_0\). Applying Proposition~\ref{prop:prescribed} gives a unique very weak solution $\mathbf{u}$ on \((0,T^*)\), together with the estimate
\[
\|\mathbf u\|_{L^s(0,T^*;L^q(\Omega))}
\le C\left(
\|\mathbf u_0\|_{L^q(\Omega)}
+\|\mathbf F\|_{L^s(0,T^*;L^r(\Omega))}
+\|\mathbf g\|_{L^s(0,T^*;W^{1-\frac1q,q}(\partial\Omega))}
\right).
\]
This proves the first part of Theorem~\ref{thm:main}, with \(T^*\) depending on the data.

For the global assertion, assume additionally that the data satisfy the smallness condition \eqref{eq:global-smallness}. By  \eqref{eq:exp}, we have
\[
\|e^{-tA_q}\mathbf u_0\|_{L^s(0,T;L^q(\Omega))}
\le C_*\|\mathbf u_0\|_{L^q(\Omega)}
\]
with \(C_*\) independent of \(T\). Consequently, for $T=\infty$,
\begin{equation}\label{global}
\mathfrak M_\infty(\mathbf u_0,\mathbf F,\mathbf g)
\le C_*\|\mathbf u_0\|_{L^q(\Omega)}
+\|\mathbf F\|_{L^s(0,\infty;L^r(\Omega))}
+\|\mathbf g\|_{L^s(0,\infty;W^{1-\frac1q,q}(\partial\Omega))}.
\end{equation}
 Choosing \(\mu\) in \eqref{eq:global-smallness} sufficiently small so that the right-hand side is bounded by \(\mu_0\), we apply Proposition~\ref{prop:prescribed}  to obtain a unique very weak solution $\mathbf{u}$, with the estimate
\[
\|\mathbf u\|_{L^s(0,\infty;L^q(\Omega))}
\le C\,\mathfrak M_\infty(\mathbf u_0,\mathbf F,\mathbf g),
\]
which is exactly \eqref{eq:global-estimate} since \eqref{global} holds. 

It remains to show that the very weak solution obtained above admits a pressure. 
Set
\[
\mathbf R:=\partial_t\mathbf u-\Delta\mathbf u+\Div(\mathbf u\otimes\mathbf u)-\Div\mathbf F\in\mathcal D'((0,T)\times\Omega).
\]
Taking a solenoidal test field $\boldsymbol\Phi\in C_c^\infty((0,T)\times\Omega)$ in \eqref{eq:smooth-transposition} gives
\[
\langle \mathbf R,\boldsymbol\Phi\rangle=0
\qquad \Div\boldsymbol\Phi=0.
\]
By the de Rham theorem, there exists $\pi\in\mathcal D'((0,T)\times\Omega)$ such that $\mathbf R+\nabla\pi=0$. The uniqueness up to a time-dependent distribution follows from the fact that a distribution with zero spatial gradient depends only on time, which ends the proof of Theorem~\ref{thm:main}. \hfill$\Box$
 

\appendix

\section{Proof of Proposition~\ref{prop:linear}}\label{app:stokes-toolbox}

We explain how the results in Breit--Gaudin \cite{BreitGaudin2025} apply to the set of exponents $\mathcal E_q=\{q,q',r,r'\}$; the parameter $\varepsilon$ in Definition~\ref{def:multiplier-boundary} is chosen sufficiently small so that the estimates cited below hold for each exponent in this set.

\textup{(i)} The Hodge decomposition, the boundedness of $\mathbb P_p$,
and the duality relation \eqref{eq:projection-duality} follow from
\cite[Theorem~4.38 and Proposition~4.40]{BreitGaudin2025}. These results
also identify $(L_\sigma^{p'}(\Omega))^*$ with $L_\sigma^p(\Omega)$ under
the integral pairing.

\textup{(ii)} We first recall the relevant results from Breit--Gaudin \cite{BreitGaudin2025}. In their work, the Stokes--Dirichlet operator, which we denote by $\widetilde A_p$, is the solenoidal Dirichlet realization of the Laplacian. At the $L^p_\sigma$-level it acts as
$$
\widetilde A_p \mathbf v=\mathbb P_p(-\Delta \mathbf v)
$$
in the distributional sense on its operator domain. \cite[Theorem~6.5]{BreitGaudin2025} gives a densely defined, invertible and $0$-sectorial operator on $L_\sigma^p(\Omega)$, with a bounded $H^\infty(\Sigma_\theta)$-functional calculus for every $\theta\in(0,\pi)$, and with the fractional-domain identification needed for the square-root property. 
In particular, these results yield invertibility, sectoriality, a bounded
$H^\infty$-functional calculus and the square-root identification
$D(\widetilde A_p^{1/2})=W_{0,\sigma}^{1,p}(\Omega)$ with the norm
equivalence in \eqref{eq:square-root}. They also give the
following variational characterization: for
$\mathbf v\in D(\widetilde A_p)$ and
$\boldsymbol\phi\in W_{0,\sigma}^{1,p'}(\Omega)$,
\begin{align}\label{eq:appendix-weak-realization}
 \int_\Omega\widetilde A_p\mathbf v\cdot\boldsymbol\phi\,\dd x
 =
 \int_\Omega\nabla\mathbf v:\nabla\boldsymbol\phi\,\dd x.
\end{align}
To prove \eqref{eq:resolvent-duality}, let $\mathbf f \in L^p_\sigma(\Omega)$ and $\boldsymbol\psi \in L^{p'}_\sigma(\Omega)$. Set $\mathbf v = A_p^{-1}\mathbf f$ and $\boldsymbol\phi = A_{p'}^{-1}\boldsymbol\psi$. Applying \eqref{eq:appendix-weak-realization} to the pair $(\mathbf v, \boldsymbol\phi)$ and then to the dual pair $(\boldsymbol\phi, \mathbf v)$ yields the desired duality identity \eqref{eq:resolvent-duality}.
The semigroup duality
\eqref{eq:semigroup-duality} follows from the corresponding duality of the
Stokes resolvents and the standard representation of the analytic semigroup.

It remains to identify $\widetilde A_p$ with the operator  $A_p$ defined by
\eqref{eq:stokes-action}. The inclusion
$D(\widetilde A_p)\subset D(A_p)$ follows from
\eqref{eq:appendix-weak-realization}. Conversely, for $\mathbf v\in D(A_p)$, set $\mathbf f=A_p\mathbf v$ and $\mathbf w:=\widetilde A_p^{-1}\mathbf f$. Then
$\mathbf v-\mathbf w\in W_{0,\sigma}^{1,p}(\Omega)$ and
\begin{align*}
 \int_\Omega\nabla(\mathbf v-\mathbf w):\nabla\boldsymbol\phi\,\dd x=0
 \qquad
 \forall\boldsymbol\phi\in W_{0,\sigma}^{1,p'}(\Omega).
\end{align*}
Using \eqref{eq:square-root-form}, this becomes
\begin{align*}
 \int_\Omega A_p^{1/2}(\mathbf v-\mathbf w)\cdot
 A_{p'}^{1/2}\boldsymbol\phi\,\dd x=0.
\end{align*}
Since $A_{p'}^{1/2}$ maps $W_{0,\sigma}^{1,p'}(\Omega)$ onto
$L_\sigma^{p'}(\Omega)$, we get $A_p^{1/2}(\mathbf v-\mathbf w)=0$. The invertibility of $A_p$ gives $\mathbf v=\mathbf w$. Hence $A_p=\widetilde A_p$.

\textup{(iii)} Since $L_\sigma^p(\Omega)$ is a complemented subspace of
$L^p(\Omega)$, it is a UMD space for $1<p<\infty$. The bounded
$H^\infty$-calculus of angle smaller than $\frac\pi2$ implies
$\mathcal R$-sectoriality, and Weis' maximal-regularity theorem stated
as \cite[Theorem~2.22]{BreitGaudin2025} yields
\eqref{eq:MR}. The constant on finite intervals is obtained by extending
the right-hand side by zero to $(0,\infty)$ and restricting the solution.

\textup{(iv)} The  bound \eqref{eq:exp} is a variant of the spectral estimate for the Stokes--Dirichlet semigroup on $L_\sigma^p(\Omega)$ in \cite[Chapter~3]{Haase2006}.
Since
\begin{align*}
\frac12+\frac32\left(\frac1{q'}-\frac1{(\frac q2)'}\right)
=\frac12+\frac{3}{2q},
\end{align*}
it follows from \cite[Proposition~6.19 and Meta-Theorem~6.20]{BreitGaudin2025} with the source space $L_\sigma^{q'}(\Omega)$ and the target space
$L^{(q/2)'}(\Omega)$ that  
\begin{align*}
\norm{\nabla e^{-tA_{q'}}\boldsymbol\psi}_{L^{(\frac q2)'}(\Omega)}
\le Ct^{-\frac12-\frac{3}{2q}}
\norm{\boldsymbol\psi}_{L^{q'}(\Omega)}
\qquad t>0.
\end{align*}
for some $C>0$.
Then by \eqref{eq:exp} we have
 \begin{equation*}
\norm{\nabla e^{-tA_{q'}}\boldsymbol\psi}_{L^{(\frac q2)'}(\Omega)}=\norm{\nabla e^{-\frac t 2 A_{q'}  }e^{-\frac t 2 A_{q'} }\boldsymbol\psi}_{L^{(\frac q2)'}(\Omega)}\\
\leq C \left(\frac t 2\right)^{-\frac12-\frac{3}{2q}}e^{-\omega_{q'}\frac t 2}\norm{\boldsymbol\psi}_{L^{q'}(\Omega)},
 \end{equation*}
which implies \eqref{eq:cross-exp}.

\bibliographystyle{plain}
\bibliography{VWS0722}

@article{Amann2003,
	author = {Amann, H.},
	title = {{Navier-Stokes} equations with nonhomogeneous {Dirichlet} data},
	journal = {J. Nonlinear Math. Phys.},
	volume = {10},
	pages = {1--11},
	year = {2003}
}

@incollection{Amann2002,
	author = {Amann, H.},
	title = {Nonhomogeneous {Navier-Stokes} equations with integrable low-regularity data},
	booktitle = {Nonlinear problems in mathematical physics and related topics},
	series ={Int. Math. Ser.},
	publisher = {Kluwer/Plenum},
	address = {New York},
	pages = {1--28},
	year = {2002}
}

@article{FarwigGaldiSohr2006a,
	author  = {Farwig, R. and Galdi, G. P. and Sohr, H.},
	title   = {A new class of weak solutions of the {Navier--Stokes} equations with nonhomogeneous data},
	journal = {J. Math. Fluid Mech.},
	volume  = {8},
	pages   = {423--444},
	year    = {2006},
	doi     = {10.1007/s00021-005-0182-6}
}

@incollection{FarwigGaldiSohr2005,
	author = {Farwig, R. and Galdi, G. P. and Sohr, H.},
	title = {Very weak solutions of stationary and instationary {Navier-Stokes} equations with nonhomogeneous data},
	booktitle = {Nonlinear elliptic and parabolic problems},
	series = {Progr. Nonlinear Differential Equations Appl.},
	volume = {64},
	publisher = {Birkh\"auser},
	address = {Basel},
	pages = {113--136},
	year = {2005}
}

@article{FarwigGaldiSohr2005b,
	author = {Farwig, R. and Galdi, G. P. and Sohr, H.},
	title = {Very weak solutions and large uniqueness classes of stationary {Navier-Stokes} equations in bounded domains of {$\mathbb{R}^2$}},
	journal = {J. Differential Equations},
	volume = {227},
	pages = {564--580},
	year = {2006}
}

@article{FarwigGaldiSohr2006b,
	author = {Farwig, R. and Kozono, H. and Sohr, H.},
	title = {Very weak solutions of the {Navier-Stokes} equations in exterior domains with nonhomogeneous data},
	journal = {J. Math. Soc. Japan},
	volume = {59},
	pages = {127--150},
	year = {2007}
}

@article{Shen2012,
	author = {Shen, Z.},
	title = {Resolvent estimates in {$L^p$} for the {Stokes} operator in {Lipschitz} domains},
	journal = {Arch. Ration. Mech. Anal.},
	volume = {205},
	pages = {395--424},
	year = {2012}
}

@article{GabelTolksdorf2022,
	author = { Gabel, F.  and  Tolksdorf, P.},
	title = {The Stokes operator in two-dimensional bounded 
		{Lipschitz} domains},
	journal = {J. Differential Equations},
	volume = {340},
	pages = {227--272},
	year = {2022},
}

@article{GengShen2024,
	author = {Geng, J. and Shen, Z.},
	title = {Resolvent estimates for the {Stokes} operator in bounded and exterior {$C^1$} domains},
	journal = {Math. Ann.},
	volume = {391},
	pages = {1467--1503},
	year = {2025}
}

@article{GengShen2024Linfty,
	author = {Geng, J. and Shen, Z.},
	title = {Resolvent estimates in {$L^\infty$} for the {Stokes} operator in nonsmooth domains},
	journal = {Invent. Math.},
	volume = {243},
	pages = {657--701},
	year = {2026}
}

@article{KunstmannWeis2017,
	author = { Kunstmann, P.C. and Weis, L. },
	title = {New criteria for the {$H^\infty$}-calculus and the {Stokes} operator on bounded {Lipschitz} domains},
	journal = {J. Evol. Equ},
	volume = {17},
	pages = {387--409},
	year = {2017},
}

@article{Tolksdorf2018,
	author = {Tolksdorf, P.},
	title = {On the {$L^p$}-theory of the {Navier-Stokes} equations on three-dimensional bounded {Lipschitz} domains},
	journal = {Math. Ann.},
	volume = {371},
	pages = {445--460},
	year = {2018}
}

@article{Coscia2017,
	author = {Coscia, V.},
	title = {Existence and uniqueness of very weak solutions to the steady-state {Navier-Stokes} problem in {Lipschitz} domains},
	journal = {J. Math. Fluid Mech.},
	volume = {19},
	pages = {819--829},
	year = {2017}
}

@article{Breit2024,
	author = {Breit, D.},
	title = {Regularity results in {2D} fluid-structure interaction},
	journal = {Math. Ann.},
	volume = {388},
	pages = {1495--1538},
	year = {2024}
}

@article{Breit2025,
	author = {Breit, D.},
	title = {Partial boundary regularity for the {Navier-Stokes} equations in irregular domains},
	journal = {J. Funct. Anal.},
	volume = {289},
	pages = {111188},
	year = {2025}
}

@book{LiebLoss2001,
	author    = {Lieb, E. H. and Loss, M.},
	title     = {Analysis},
	series    = {Graduate Studies in Mathematics},
	volume    = {14},
	edition   = {Second},
	publisher = {American Mathematical Society},
	address   = {Providence, RI},
	year      = {2001},
	isbn      = {978-0-8218-2783-3}
}

@article{BreitGaudin2025,
	author        = {Breit, D. and Gaudin, A.},
	title         = {Optimal regularity results for the {Stokes--Dirichlet} problem},
	journal       = {arXiv preprint},
	year          = {2025, 10.48550/arXiv.2511.19091},
	eprint        = {2511.19091},
	archiveprefix = {arXiv},
	primaryclass  = {math.AP},
	doi           = {10.48550/arXiv.2511.19091}
}

@book{MS09,
	author = {Maz'ya, V. G. and Shaposhnikova, T. O.},
	title = {Theory of {Sobolev} multipliers: With applications to differential and integral operators},
	series = {Grundlehren der mathematischen Wissenschaften},
	volume = {337},
	publisher = {Springer-Verlag},
	address = {Berlin},
	year = {2009}
}

@book{Haase2006,
	author    = {Haase, M.},
	title     = {The Functional Calculus for Sectorial Operators},
	series    = {Operator Theory: Advances and Applications},
	volume    = {169},
	publisher = {Birkh{\"a}user},
	address   = {Basel},
	year      = {2006},
	doi       = {10.1007/3-7643-7698-8}
}

@book{Galdi2011,
	author    = {Galdi, G. P.},
	title     = {An Introduction to the Mathematical Theory of the {Navier--Stokes} Equations: Steady-State Problems},
	series    = {Springer Monographs in Mathematics},
	edition   = {Second},
	publisher = {Springer},
	address   = {New York},
	year      = {2011},
	doi       = {10.1007/978-0-387-09620-9}
}

@article{Giga1985Domain,
  author    = {Giga, Y.},
  title     = {Domains of fractional powers of the {S}tokes operator in {$L_r$} spaces},
  journal   = {Arch. Ration. Mech. Anal.},
  volume    = {89},
  pages     = {251--265},
  year      = {1985}
}

@article{Giga1986Semilinear,
  author    = {Giga, Y.},
  title     = {Solutions for semilinear parabolic equations in {$L^p$} and regularity of weak solutions of the {N}avier--{S}tokes system},
  journal   = {J. Differential Equations},
  volume    = {62},
  pages     = {186--212},
  year      = {1986}
}

@article{GigaMiyakawa1985,
  author    = {Giga, Y. and Miyakawa, T.},
  title     = {Solutions in {$L_r$} of the {N}avier--{S}tokes initial value problem},
  journal   = {Arch. Ration. Mech. Anal.},
  volume    = {89},
  pages     = {267--281},
  year      = {1985}
}

@article{GigaSohr1991,
  author    = {Giga, Y. and Sohr, H.},
  title     = {Abstract {$L^p$} estimates for the {C}auchy problem with applications to the {N}avier--{S}tokes equations in exterior domains},
  journal   = {J. Funct. Anal.},
  volume    = {102},
  pages     = {72--94},
  year      = {1991}
}





\end{document}